\documentclass[12pt]{article}

\usepackage[dvipdfmx]{graphicx}
\usepackage{amsmath}
\usepackage{amssymb}
\usepackage{amsthm}
\usepackage{bm}

\newtheorem{theorem}{Theorem}[section]
\newtheorem{proposition}[theorem]{Proposition}
\newtheorem{corollary}[theorem]{Corollary}
\newtheorem{lemma}[theorem]{Lemma}
\theoremstyle{definition}
\newtheorem{remark}{Remark}[section]
\newtheorem{example}{Example}[section]
\numberwithin{equation}{section}
\newcommand{\RR}{{\mathbb{R}}}
\newcommand{\ZZ}{{\mathbb{Z}}}
\newcommand{\argmin}{\mathop{\rm argmin}\limits}
\DeclareMathOperator*{\argmax}{\mathrm{argmax}}
\newcommand{\dom}{\mathop{\rm dom}\limits}
\newcommand{\dmid}{\mu}
\newcommand{\vecone}{\bm{1}}
\newcommand{\veczero}{\bm{0}}
\newcommand{\finbox}{\hspace*{\fill}$\rule{0.2cm}{0.2cm}$}

\begin{document}

\title{Directed Discrete Midpoint Convexity}

\author{Akihisa Tamura\thanks{
Akihisa Tamura:
Dept. of Math., Keio University, Japan,
\texttt{aki-tamura@math.keio.ac.jp}
}
\and Kazuya Tsurumi\thanks{
Kazuya Tsurumi:
Dept. of Math., Keio University, Japan,
\texttt{goro56koropappa0@keio.jp}
}}

\date{20 January 2020}

\maketitle

\begin{abstract}
For continuous functions, midpoint convexity characterizes convex
functions.
By considering discrete versions of midpoint convexity, several types of
discrete convexities of functions, including integral convexity, 
L$^\natural$-convexity and global/local discrete midpoint convexity, 
have been studied.
We propose a new type of discrete midpoint convexity that lies between 
L$^\natural$-convexity and integral convexity and is independent of
global/local discrete midpoint convexity.
The new convexity, named DDM-convexity, has nice properties
satisfied by L$^\natural$-convexity and 
global/local discrete midpoint convexity.
DDM-convex functions are stable under scaling, satisfy 
the so-called parallelgram inequality and a proximity theorem
with the same small proximity bound as
that for L$^{\natural}$-convex functions.
Several characterizations of DDM-convexity are given and
algorithms for DDM-convex function minimization are developed.
We also propose DDM-convexity in continuous variables and 
give proximity theorems on these functions.
\end{abstract}


\section{Introduction}\label{intro}

For a continuous function $f$ defined on a convex set 
$S \subseteq\RR^n$, it was proved by Jensen~\cite{Jen0506} that
{\em midpoint convexity} defined by
\[
 f(x) + f(y) \geq 2 f\left(\frac{x+y}{2}\right) \qquad
 (\forall x,y \in S)
\]
is equivalent to the inequality defining convex functions
\[
 \alpha f(x) + (1-\alpha)f(y) \geq f(\alpha x + (1-\alpha) y) \qquad
 (\forall x,y \in S;\; \forall \alpha \in [0,1]).
\]
By capturing the concept of midpoint convexity, several types of
`discrete' midpoint convexities for functions defined 
on the integer lattice $\ZZ^n$ have been proposed.

A weak version of `discrete' midpoint convexity is obtained by
replacing $f((x+y)/2)$ by the smallest value of a linear extension of
$f$ among the integer points neighboring $(x+y)/2$.
More precisely, for any point $x \in \RR^n$, we consider
its integer neighborhood
\[
 N(x) = \{ z \in \ZZ^n \mid |z_i-x_i| < 1 \;(i=1,,\ldots,n)\}
\]
and the set $\Lambda(x)$ of all coefficients
$(\lambda_z \mid z \in N(x))$ 
for convex combinations indexed by $N(x)$.
For a function $f:\ZZ^n \to \RR \cup \{+\infty\}$, 
we define the local convex envelope $\tilde{f}$ of $f$ by
\[
 \tilde{f}(x) = \min\left\{ \sum_{z \in N(x)}\lambda_z f(z) \mid
  \sum_{z \in N(x)}\lambda_z z = x,\; (\lambda_z) \in \Lambda(x)\right\}
  \qquad (x \in \RR^n).
\]
We say that $f$ satisfies {\em weak discrete midpoint convexity} if
the following inequality holds
\begin{equation}\label{weakDMC}
 f(x) + f(y) \geq 2 \tilde{f} \left(\frac{x+y}{2}\right)
\end{equation}
for all $x,y \in \ZZ^n$.
On the other hand, 
$f$ is said to
be {\em integrally convex} \cite{favati-tardella}
if $\tilde{f}$ is convex on $\RR^n$.
Characterizations of integral convexity by using 
weak discrete midpoint convexity
have been discussed in \cite{favati-tardella,MMTT2019,MMTT-DMC}.
The simplest characterization, Theorem A.1 in \cite{MMTT-DMC}, says that 
$f$ is integrally convex if and only if $f$ satisfies (\ref{weakDMC}) 
for all $x,y \in \dom f$ with\footnote{
If $\|x-y\|_{\infty} \leq 1$, then (\ref{weakDMC}) obviously holds.}
$\|x-y\|_{\infty} \geq 2$,
where the effective domain $\dom f$ of $f$ is defined by
\[
 \dom f = \{ x \in \ZZ^n \mid f(x) < +\infty\}.
\]

The class of integrally convex functions establishes a general framework
of discrete convex functions, including separable convex, 
L$^\natural$-convex, M$^\natural$-convex, L$^\natural_2$-convex,
M$^\natural_2$-convex functions \cite{DCA},
BS-convex and UJ-convex functions \cite{Fuj2014}, and
globally/locally discrete midpoint convex functions \cite{MMTT-DMC}.
The concept of integral convexity is used in formulating
discrete fixed point theorems \cite{Iim10,IMT05,Yan09fixpt},
designing algorithms for discrete systems of 
nonlinear equations \cite{LTY11nle,Yan08comp}, and
guaranteeing the existence of a pure strategy equilibrium in
finite symmetric games \cite{IW14}.

A strong version of `discrete' midpoint convexity is obtained by
replacing $f((x+y)/2)$ by the average of the values of $f$ at two
integer points obtained by rounding-up and rounding-down of 
all components of $(x+y)/2$.
More precisely, 
for a function $f:\ZZ^n \to \RR \cup \{+\infty\}$, 
we say that $f$ satisfies {\em discrete midpoint convexity} if 
it has
\begin{equation}\label{DMC}
 f(x) + f(y) \geq 
   f\left(\left\lceil  \frac{x+y}{2}  \right\rceil\right) +
   f\left(\left\lfloor  \frac{x+y}{2}  \right\rfloor\right)
\end{equation}
for all $x,y \in \ZZ^n$, where $\lceil \cdot \rceil$ and
$\lfloor \cdot \rfloor$ denote the integer vectors obtained by
rounding up and rounding down all components of a given real vector,
respectively.
It is known that discrete midpoint convexity characterizes
the class of L$^\natural$-{\em convex functions} \cite{FM00,Mdca98}
which play important roles in both theoretical and practical aspects.
L$^\natural$-convex functions are applied to several fields, including
auction theory~\cite{LLM06,MSY16auction},
image processing~\cite{KS09lnatmin},
inventory theory~\cite{Chen2017,SCB14,Zip08} and
scheduling~\cite{BQ11}.
Since discrete midpoint convexity (\ref{DMC}) obviously implies
weak discrete midpoint convexity (\ref{weakDMC}), 
L$^\natural$-convex functions forms a subclass of
integrally convex functions.

Moriguchi et al.~\cite{MMTT-DMC} classified discrete convex functions
between L$^\natural$-convex and integrally convex functions in terms of
discrete midpoint convexity with $\ell_\infty$-distance requirements, 
and proposed two new classes of discrete convex functions, 
namely, globally/locally discrete midpoint convex functions.
A function $f:\ZZ^n \to \RR \cup \{+\infty\}$ is said to
be {\em globally discrete midpoint convex} if (\ref{DMC}) holds
for any pair $(x,y) \in \ZZ^n \times \ZZ^n$ with
$\|x-y\|_\infty \geq 2$.
A set $S \subseteq \ZZ^n$ is called
a {\em discrete midpoint convex set} if its indicator function $\delta_S$
defined by 
\[
\delta_S(x)=
\begin{cases}
0 & (x\in S)\\
+\infty & (x\not\in S)
\end{cases}
\qquad
(x \in \ZZ^n)
\]
is globally discrete midpoint convex, that is, if
\[
x, y\in S \mbox{ with } \|x-y\|_\infty \geq 2 \;\Rightarrow\;  
\left\lceil  \frac{x+y}{2}  \right\rceil,\; 
\left\lfloor  \frac{x+y}{2}  \right\rfloor \in S.
\]
A function $f:\ZZ^n \to \RR \cup \{+\infty\}$ is said to
be {\em locally discrete midpoint convex} if $\dom f$ is 
a discrete midpoint convex set and (\ref{DMC}) holds
for any pair $(x,y) \in \ZZ^n \times \ZZ^n$ with
$\|x-y\|_\infty = 2$.
It is shown in \cite{MMTT-DMC} that the following inclusion relations 
among function classes hold:
\begin{align*}
& \{ \mbox{ ${\rm L}\sp{\natural}$-convex } \}
\ \subsetneqq \
\{ \mbox{ globally discrete midpoint convex  } \} \\
& \subsetneqq \
\{ \mbox{ locally discrete midpoint convex } \}
\ \subsetneqq \
\{ \mbox{ integrally convex } \},
\end{align*}
and globally/locally discrete midpoint convex functions inherit 
nice features from L$^\natural$-convex functions, that is,
for a globally/locally discrete midpoint convex $f$ and 
a positive integer $\alpha$,
\begin{itemize}
\item
the scaled function $f^\alpha$ defined by 
$f^\alpha(x) = f(\alpha x)\; (x \in \ZZ^n)$
belongs to the same class, that is,
global/local discrete midpoint convexity is closed
with respect to scaling operations,
\item a proximity theorem with the same proximity distance with 
L$^\natural$-convexity holds, that is, given an $x^\alpha$ with
$f(x^\alpha) \leq f(x^\alpha + \alpha d)$ for all $d \in \{-1,0,+1\}^n$,
there exists a minimizer $x^*$ of $f$ with 
$\|x^\alpha -x ^*\|_\infty \leq n(\alpha-1)$,
\item when $f$ has a minimizer,
a steepest descent algorithm for the minimization of $f$ is developed 
such that the number of local minimizations in the neighborhood of
$\ell_\infty$-distance $2$ (the $2$-neighborhood minimizations) 
is bounded by the shortest $\ell_\infty$-distance 
from a given initial feasible point to a minimizer of $f$, and
\item when $\dom f$ is bounded and $K_\infty$ denotes
the $\ell_\infty$-size of $\dom f$, a scaling algorithm
minimizing $f$ with $O(n \log_2 K_\infty$) calls of
the $2$-neighborhood minimization is developed.
\end{itemize}

This paper, strongly motivated by \cite{MMTT-DMC}, proposes
a new type of discrete midpoint convexity between L$^\natural$-convexity
and integral convexity, but it is independent of
global/local discrete midpoint convexity with respect to inclusion relation.
We name the new convexity 
{\em directed discrete midpoint convexity (DDM-convexity)} which forms
the following classification
\[
 \{ \mbox{ ${\rm L}\sp{\natural}$-convex } \}
\ \subsetneqq \
\{ \mbox{ DDM-convex  } \} 
 \ \subsetneqq \
\{ \mbox{ integrally convex } \}.
\]
The same features as mentioned above are satisfied by DDM-convexity.
The merits of DDM-convexity relative to 
global/local discrete midpoint convexity are the following properties:
\begin{itemize}
\item 
DDM-convexity is closed with respect to individual sign inversion of 
variables, that is, for a DDM-convex function $f$ and 
$\tau_i \in \{+1,-1\}\;(i=1,\ldots,n)$, 
$f(\tau_1 x_1,\ldots,\tau_n x_n)$ is also DDM-convex (see
Proposition~\ref{prop:basic-operations} (3)).
Neither L$^\natural$-convexity nor 
global nor local discrete midpoint convexity has this property,
while integral convexity is closed with respect to individual sign inversion
of variables.
\item
For a quadratic function $f(x) = x^\top Q x$ with a symmetric matrix
$Q = [q_{ij}]$, DDM-convexity is characterized by 
the diagonal dominance with nonnegative diagonals of $Q$:
\[
 q_{ii} \geq \sum_{j:j \neq i}|q_{ij}| \qquad (\forall i =1,\ldots,n)
\]
(see Theorem~\ref{thm:diagdom-DDMC}).
While L$^\natural$-convexity is characterized by the combination of 
diagonal dominance with nonnegative diagonals and 
nonpositivity of all off-diagonal components of $Q$, 
global/local discrete midpoint convexity is independent of
the diagonal dominance with nonnegative diagonals.
\item
A function $g : \ZZ \to \RR \cup \{+\infty\}$ is said to be discrete convex
if $g(t-1)+g(t+1) \geq 2g(t)$ for all $t \in \ZZ$.
For univariate discrete convex functions
$\xi_{i}, \varphi_{ij}, \psi_{ij} : 
\ZZ \to \RR \cup \{+\infty\}\; 
(i=1,\ldots,n; j \in \{1,\ldots,n\} \setminus \{i\}$),
a {\em 2-separable convex function}~\cite{L-extendable} is defined as
a function represented as
\[
 f(x)=\sum_{i=1}^n \xi_i(x_i)+
 \sum_{i,j:j\not=i}\varphi_{ij} (x_i-x_j)+
 \sum_{i,j:j\not=i} \psi_{ij}(x_i+x_j)\qquad (x\in\ZZ^n).
\]
The class of DDM-convex functions includes 
all 2-separable convex functions (see Theorem~\ref{thm:2sep2ddmc}).
It is known that if all $\psi_{ij}$ are identically zero, then $f$ is
L$^\natural$-convex, whereas
there exists a 2-separable convex function not contained
in the class of globally/locally discrete midpoint convex functions.
\item
A steepest descent algorithm for the minimization of 
DDM-convex functions requires only the 1-neighborhood minimization
in contrast to the 2-neighborhood minimization
(see Section~\ref{sec:descentalgo}).
\end{itemize}

In the next section, we give the definition of 
DDM-convexity and basic properties of DDM-convex functions.
In Section~\ref{relation-sec}, we discuss a relationship between
DDM-convexity and known discrete convexities.
 For globally/locally discrete midpoint convex functions, 
Moriguchi et al.~\cite{MMTT-DMC} revealed a useful property, 
which is expressed by the so-called parallelgram inequality.
We show that a similar parallelgram inequality holds 
for DDM-convex functions in Section~\ref{para-ineq}.
Sections~\ref{chars} and \ref{operations} are devoted to
characterizations and operations for DDM-convexity.
We prove a proximity theorem for DDM-convex functions 
in Section~\ref{proximity}, while in Section~\ref{min-algo}
we propose a steepest descent algorithm and a scaling algorithm
for DDM-convex function minimization.
In Section~\ref{continuousDDMC}, we define DDM-convex functions 
in continuous variables and give proximity theorems
for such functions.

\section{Directed discrete midpoint convexity}\label{def-DDMC-section}

We give the definition of directed discrete midpoint convexity and  show
its basic properties.

For an ordered pair $(x,y)$ of $x,y\in\ZZ^n$, 
we define $\dmid(x,y) \in \ZZ^n$ by
\[
\dmid(x,y)_i=
\begin{cases}
    \displaystyle \left\lceil \frac{x_i+y_i}{2} \right\rceil & (x_i\ge y_i), \\[1em]
    \displaystyle \left\lfloor \frac{x_i+y_i}{2} \right\rfloor & (x_i<y_i).
\end{cases}
\]
That is, each component $\dmid(x,y)_i$ of $\dmid(x,y)$ is defined 
by rounding up or rounding down $\frac{x_i+y_i}{2}$ to the integer
in the direction of $x_i$.
It is easy to show the next characterization of $\dmid(x,y)$ and
$\dmid(y,x)$.

\begin{proposition}\label{prop:dmid}
For $x,y,p,q\in\ZZ^n$, $p=\dmid(x,y)$ and $q=\dmid(y,x)$ hold
if and only if the following conditions {\rm (a)}$\sim${\rm (c)} hold:
\begin{enumerate}
\renewcommand{\labelenumi}{\rm (\alph{enumi})}
\item $p+q = x+y$,
\item $\|p-q\|_\infty\le 1$, and
\item for each $i=1,\dots,n$, if $x_i\ge y_i$, then $p_i\ge q_i$; 
otherwise $p_i\le q_i$.
\end{enumerate}
\end{proposition}

For every $a,b\in\RR^n$, let us denote 
the $n$-dimensional vector $(a_1b_1,\dots,a_nb_n)$ by $a\odot b$.
The next proposition gives fundamental properties of $\dmid(\cdot,\cdot)$.

\begin{proposition}\label{prop:basic}
For every $ x,y,d\in\ZZ^n$, the following properties hold.
\begin{enumerate}
\renewcommand{\labelenumi}{\rm (\arabic{enumi})}
\item $x+y=\dmid(x,y)+\dmid(y,x)$.
\item If $x\ge y$, then $\dmid(x,y)=\lceil\frac{x+y}{2}\rceil$
and $\dmid(y,x)= \lfloor\frac{x+y}{2}\rfloor$.
\item If $\|x-y\|_\infty\le 1$, then $\dmid(x,y)=x$ and $\dmid(y,x)=y$.
\item $\dmid(x+d,y+d)=\dmid(x,y)+d$.
\item For any permutation $\sigma$ of $(1,\dots,n)$,
\[
 \dmid((x_{\sigma(1)},\dots,x_{\sigma(n)}),
 (y_{\sigma(1)},\dots,y_{\sigma(n)}))=
 (\dmid(x,y)_{\sigma(1)},\dots,\dmid(x,y)_{\sigma(n)}).
\]
\item For any $\tau\in\{+1,-1\}^n$, 
$\dmid(\tau\odot x,\tau \odot y)=\tau\odot\dmid(x,y)$.
\end{enumerate}
\end{proposition}
\begin{proof}
Properties (1)-(5) are obvious by the definition of $\dmid(\cdot,\cdot)$.
Let us show (6). If $\tau_i=+1$,
\[
\dmid(\tau\odot x,\tau \odot y)_i=
\begin{cases}
\lceil\frac{x_i+y_i}{2}\rceil=\dmid(x,y)_i&\ (x_i\ge y_i)\\
\lfloor \frac{x_i+y_i}{2}\rfloor=\dmid(x,y)_i&\ (x_i<y_i)
\end{cases}
\]
holds, and if $\tau_i=-1$,
\[
\dmid(\tau\odot x,\tau \odot y)_i=
\begin{cases}
\lfloor\frac{-x_i-y_i}{2}\rfloor=-\lceil\frac{x_i+y_i}{2}\rceil=-\dmid(x,y)_i&\ (x_i \geq y_i)\\
\lceil\frac{-x_i-y_i}{2}\rceil=-\lfloor\frac{x_i+y_i}{2}\rfloor=-\dmid(x,y)_i&\ (x_i < y_i)
\end{cases}
\]
holds.
\end{proof}

By using the introduced $\dmid(\cdot,\cdot)$, we propose
new classes of functions and sets.
We say that a function $f:\ZZ^n \to \RR \cup \{+\infty\}$
satisfies {\em directed discrete midpoint convexity (DDM-convexity)}\
or is 
a {\em directed discrete midpoint convex function (DDM-convex function)}\ if
\begin{equation}\label{DDMC}
f(x)+f(y)\ge f(\dmid(x, y))+f(\dmid(y, x)) 
\end{equation}
for all  $x,y \in \ZZ^n$.
We call $S \subseteq \ZZ^n$
a {\em directed discrete midpoint convex set (DDM-convex set)}\
if its indicator function $\delta_S$ is DDM-convex, that is, if
\[
x, y\in S\Rightarrow  \dmid(x, y),\; \dmid(y, x)\in S
\]
holds.

The next propositions are direct consequences of 
Proposition~\ref{prop:basic} and the definition (\ref{DDMC}).

\begin{proposition}\label{prop:obviousprop} The following statements hold:

\begin{enumerate}
\renewcommand{\labelenumi}{\rm (\arabic{enumi})}
\item Any function defined on $\{0, 1\}^n$ is a DDM-convex function.
\item Any subset of $\{0, 1\}^n$ is a DDM-convex set.
\item For a DDM-convex function $f$, its effective domain $\dom f$ and 
the set $\argmin f$ of minimizers of $f$ are DDM-convex sets, where
$\argmin f$ is defined by
\[
 \argmin f = \{ x \in \ZZ^n \mid f(x) \leq f(z)\;
 (\forall z \in \ZZ^n)\}.
\]
\end{enumerate}
\end{proposition}

\begin{proposition}\label{prop:basic-operations}
Let $f, f_1, f_2:\ZZ^n \to \RR \cup \{+\infty\}$ be
DDM-convex functions.
\begin{enumerate}
\renewcommand{\labelenumi}{\rm (\arabic{enumi})}
\item For any $d\in\ZZ^n$, 
  $g(x)=f(x+d)$ is a DDM-convex function.
\item For any permutation $\sigma$ of $(1, \dots, n)$,
  $g(x)=f(x_{\sigma(1)}, \dots, x_{\sigma(n)})$ is a DDM-convex function.
\item For any $\tau\in\{+1, -1\}^n$, 
  $g(x)=f(\tau\odot x)$ is a DDM-convex function.
\item For any $a_1, a_2\ge0$, 
  $g(x)=a_1f_1(x)+a_2f_2(x)$ is a DDM-convex function.
\end{enumerate}
\end{proposition}

\begin{proposition}\label{prop:basic-operationsset}
Let $S, S_1, S_2\subseteq\ZZ^n$ be DDM-convex sets.
\begin{enumerate}
\renewcommand{\labelenumi}{\rm (\arabic{enumi})}
\item For any $d\in\ZZ^n$, 
  $T=\{x+d \mid x\in S\}$ is a DDM-convex set.
\item For any permutation $\sigma$ of $(1, \dots, n)$,
  $T=\{(x_{\sigma(1)},\dots,x_{\sigma(n)}) \mid (x_1,\dots,x_n) \in S \}$
  is a DDM-convex set.
\item For any $\tau\in\{+1, -1\}^n$, 
  $T=\{\tau\odot x \mid  x\in S\}$ is a DDM-convex set.
\item $T=S_1\cap S_2$ is a DDM-convex set.
\end{enumerate}
\end{proposition}

\section{Relationships with known discrete convexities}\label{relation-sec}

We discuss relationships between DDM-convexity and known discrete
convexities, including integral convexity, L$^\natural$-convexity, 
global/local discrete midpoint convexity and 2-separable convexity.

As mentioned in Section~\ref{intro}, 
the class of integrally convex functions is characterized by
weak discrete midpoint convexity (\ref{weakDMC}).
Since DDM-convexity (\ref{DDMC}) trivially implies (\ref{weakDMC}), 
any DDM-convex function is integrally convex.
Therefore, DDM-convex functions inherit many properties of 
integrally convex functions.
We introduce a good property of integrally convex functions as well as
DDM-convex functions, {\em box-barrier property}.

\begin{theorem}
[\protect{Box-barrier property~\cite[Theorem 2.6]{MMTT2019}}]\label{box-barrier}
Let $f:\ZZ^n \to \RR \cup \{+\infty\}$ be
an integrally convex function, and let 
$p\in(\ZZ\cup\{-\infty\})^n$ and 
$q\in(\ZZ\cup\{+\infty\})^n$ with $p\leq q$.
Define 
\begin{align*}
S  &=\{x\in\ZZ^n \mid p_i<x_i<q_i\ (i=1, \dots, n)\}, \\
W^+_i &=\{x\in\ZZ^n \mid x_i=q_i, p_j\le x_j\le q_j\ (j\not=i)\}\quad (i=1, \dots, n),\\
W^-_i &=\{x\in\ZZ^n \mid x_i=p_i, p_j\le x_j\le q_j\ (j\not=i)\}\quad (i=1, \dots, n), \\
W &=\bigcup_{i=1}^n (W^+_i\cup W^-_i),
\end{align*}
and $\hat{x}\in S\cap \dom f$.
If $f(\hat{x})\leq f(y)$ for all $y\in W$, then
$f(\hat{x})\le f(z)$ for all $z\in\ZZ^n\setminus S$.
\end{theorem}
By setting $p = \hat{x} - \vecone$ and $q = \hat{x} + \vecone$ where
$\vecone$ denotes the vector of all ones, box-barrier property implies
the minimality criterion of integrally convex functions.

\begin{theorem}[\protect{\cite[Proposition 3.1]{favati-tardella}; 
see also \cite[Theorem 3.21]{DCA}}]\label{th:1-opt}
Let $f: \ZZ^n \to \RR \cup \{+\infty\}$
be an integrally convex function and $\hat{x} \in \dom f$.
Then $\hat{x}$ is a minimizer of $f$
if and only if
$f(\hat{x}) \leq f(\hat{x} +  d)$ for all 
$d \in  \{ -1, 0, +1 \}^{n}$.
\end{theorem}

As a special case of Theorem~\ref{th:1-opt}, we have
the minimality criterion of DDM-convex functions.

\begin{corollary}\label{col:1-opt}
Let $f: \ZZ^n \to \RR \cup \{+\infty\}$
be a DDM-convex function and $\hat{x} \in \dom f$.
Then $\hat{x}$ is a minimizer of $f$
if and only if
$f(\hat{x}) \leq f(\hat{x} +  d)$ for all 
$d \in  \{ -1, 0, +1 \}^{n}$.
\end{corollary}

We next discuss the relationship between L$^\natural$-convexity
and DDM-convexity.
L$^\natural$-convex functions are originally defined by
{\em translation-submodularity}:
\begin{equation}\label{trans-submo}
 f(x)+f(y)\geq f((x-\alpha\vecone)\vee y)+f(x\wedge (y+\alpha \vecone))
\end{equation}
for all $x,y \in \ZZ^n$ and nonnegative integer $\alpha$, where
$p \vee q$ and $p \wedge q$ denote the componentwise maximum and minimum
of the vectors $p$ and $q$, respectively.
Translation-submodularity is a generalization of {\em submodularity}:
\begin{equation}\label{submo}
 f(x)+f(y)\geq f(x \vee y)+f(x \wedge y).
\end{equation}
L$^\natural$-convexity has several equivalent characterizations as below.
\begin{theorem}[\protect{\cite[Corollary 5.2.2]{favati-tardella}, 
\cite[Theorem 3]{FM00}, \cite[Theorem 7.7]{DCA}}]\label{thm:L-char}
For a function $f:\ZZ^n \to \RR \cup \{+\infty\}$,
the following properties are equivalent:
\begin{enumerate}
\renewcommand{\labelenumi}{\rm (\arabic{enumi})}
\item 
$f$ is $\textrm{L}^\natural$-convex, that is, $(\ref{trans-submo})$ holds
for all $x,y \in \ZZ^n$ and nonnegative integer $\alpha$.
\item 
$f$ satisfies discrete midpoint convexity $(\ref{DMC})$
for all $x,y \in \ZZ^n$.
\item $f$ is integrally convex and submodular.
\item For every $x, y\in\ZZ^n$ with $x \not\geq y$ and 
$A=\argmax_i\{y_i-x_i\}$,
\[
    f(x)+f(y)\geq f(x+\vecone_A)+f(y-\vecone_A),
\]
where the $i$-component of $\vecone_A$ is $1$ if $i \in A$; otherwise $0$.
\end{enumerate}
\end{theorem}
Theorem~\ref{thm:L-char} yields the next property.
\begin{proposition}\label{prop:LtoDDMC}
Any L$^\natural $-convex function is DDM-convex.
\end{proposition}
\begin{proof}
Let $f : \ZZ \to \RR \cup \{+\infty\}$ be an L$^\natural$-convex function.
We arbitrarily fix $x,y \in \dom f$ and show that (\ref{DDMC}) holds for $x$ and $y$.
By Proposition~\ref{prop:basic}~(3), (\ref{DDMC}) holds if $\|x-y\|_\infty \leq 1$.
Suppose that $m = \|x-y\|_\infty \geq 2$ and $\|y-x\|_\infty=\max_i \{y_i-x_i\}$ 
(by exchanging the role of $x$ and $y$ if necessary).
Let $A = \argmax_i\{y_i - x_i\}$, $p=x+\vecone_A$ and $q=y-\vecone_A$.
The vectors $p$ and $q$ satisfy the properties (a) and (c) of
Proposition~\ref{prop:dmid} and $\|p-q\|_\infty \leq m$.
Furthermore,  by Theorem~\ref{thm:L-char}~(4), we have
\begin{equation}\label{eq1:lem:LtoDDMC}
 f(x) + f(y) \geq f(p) + f(q).
\end{equation}
If $\|p-q\|_\infty=m$, by applying the above process for the pair $(q,p)$ again,
we obtain $p$ and $q$ with $\|p-q\|_\infty < m$ preserving 
(a), (c) of Proposition~\ref{prop:dmid} and (\ref{eq1:lem:LtoDDMC}).
By repeating this argument, we finally obtain $p$ and $q$ having
(a)$\sim$(c) of Proposition~\ref{prop:dmid} and (\ref{eq1:lem:LtoDDMC}), which
means that (\ref{DDMC}) holds for $x$ and $y$.
\end{proof}

A set $S \subseteq \ZZ^n$ is called an L$^\natural$-convex set if
its indicator function $\delta_S$ is L$^\natural$-convex.

\begin{corollary}
Any L$^\natural$-convex set is DDM-convex.
\end{corollary}

\begin{example}({\cite[Remark 1] {MMTT-DMC}})
The class of L$^\natural$-convex functions is a proper subclass of 
DDM-convex functions.
For example, 
\[
S=\{(1, 0), (0, 1)\}
\]
is a DDM-convex set, but for $x=(1, 0), y=(0, 1)$, 
$\lceil\frac{x+y}{2}\rceil=(1, 1)\not\in S$
and $\lfloor\frac{x+y}{2}\rfloor=(0, 0)\not\in S$, which means that $S$ 
is not L$^\natural$-convex. 
\hfill\finbox
\end{example}

\begin{example}({\cite[Remark 2] {MMTT-DMC}})
A set $S\subseteq\ZZ^n$ is said to be {\em L$^\natural_2$-convex set} 
if it is the Minkowski sum of two L$^\natural$-convex sets.
DDM-convexity and L$^\natural_2$-convexity are mutually independent.
For example,
\[
\{(1, 0), (0, 1)\}
\]
is a DDM-convex set, but is not L$^\natural_2$-convex.
On the other hand,
\[
S=\{(0, 0, 0, 0), (0, 1, 1, 0), (1, 1, 0, 0), (1, 2, 1, 0)\}
\]
is the Minkowski sum of two L$^\natural$-convex sets 
$S_1 = \{\{(0, 0, 0, 0), (0, 1, 1, 0)\}$ and $S_2 = \{(0, 0, 0, 0),(1, 1, 0, 0)\}$, 
but $S$ is not DDM-convex because for $x=(0, 0, 0, 0)$, $y=(1, 2, 1, 0)$,
$\dmid(x, y)=(0, 1, 0, 0)\not\in S$ and $\dmid(y, x)=(1, 1, 1, 0)\not\in S$.
\hfill\finbox
\end{example}

We next discuss the independence between 
global/local discrete midpoint convexity and DDM-convexity
by showing the independence between discrete midpoint convex sets and
DDM-convex sets.

\begin{example}\label{DMCvsDDMC}
It is easy to show that the set $S$ defined by
\[
S=\{(0, 0, 0), (1, 1, 0), (1, 0, -1), (2, 1, -1)\}
\]
is discrete midpoint convex, but $S$ is not DDM-convex because for
$x=(0, 0, 0)$ and $y=(2, 1, -1)$, we have $\dmid(x, y)=(1, 0, 0)\not\in S$ and
$\dmid(y, x)=(1, 1, -1)\not\in S$.

On the other hand, 
\[
T=\{(0, 0, 0), (1, 0, 0), (1, 1, 1), (2, 1, 1), (1, 1, -1), (2, 1, -1), (1, 1, 0), (2, 1, 0)\}
\]
is DDM-convex.
However, $T$ is not discrete midpoint convex, and moreover,
for any $(\tau_1, \tau_2, \tau_3)\in\{-1, +1\}^3$, the modified set
\[
 \tau \odot T = \{(\tau_1x_1, \tau_2x_2, \tau_3x_3) \mid (x_1, x_2, x_3)\in T\}
\]
is not discrete midpoint convex while it is DDM-convex 
by Proposition~\ref{prop:basic-operationsset}~(3).
The reason is as follows.
Since $T$ is symmetric on the third component, we can assume $\tau_3=+1$.
\begin{itemize}
\item
In the case where $\tau = (\pm 1, +1, +1)$, 
for $x=(0, 0, 0) = \tau \odot (0,0,0)$ and 
$y=(\pm 2, 1, -1)=\tau \odot (2,1,-1)$, 
$\lfloor (x+y)/2 \rfloor = (\pm 1, 0, -1)\not\in \tau \odot T$.
\item
In the case where $\tau = (\pm 1, -1, +1)$,
for $x = (0, 0, 0) = \tau \odot (0,0,0)$ and 
$y = (\pm 2, -1, 1)=\tau \odot (2,1,1)$,
 $\lceil (x+y)/2 \rceil = (\pm 1, 0, 1)\not\in \tau \odot T$.
 \hfill\finbox
\end{itemize}
\end{example}

We finally show that 2-separable convex functions are DDM-convex.
Let $\xi_{i}, \varphi_{ij}, \psi_{ij} : \ZZ \to \RR \cup \{+\infty\}\; 
(i=1,\ldots,n; j \in \{1,\ldots,n\} \setminus \{i\}$) 
be univariate discrete convex functions.
A {\em 2-separable convex function}~\cite{L-extendable} is defined as
a function represented as
\begin{equation}\label{def:2separable}
 f(x)=\sum_{i=1}^n \xi_i(x_i)+
 \sum_{i,j:j\not=i}\varphi_{ij} (x_i-x_j)+
 \sum_{i,j:j\not=i} \psi_{ij}(x_i+x_j)\qquad (x\in\ZZ^n).
\end{equation}
It is known that the function $g$ defined by
\[
 g(x)=\sum_{i=1}^n \xi_i(x_i)+
 \sum_{i,j:j\not=i}\varphi_{ij} (x_i-x_j) \qquad (x\in\ZZ^n)
\]
is L$^\natural$-convex~\cite[Proposition 7.9]{DCA}.
By Proposition~\ref{prop:basic-operations}~(4) and 
Proposition~\ref{prop:LtoDDMC},
it is enough to show that each $\psi_{ij}$ is DDM-convex in order to prove 
DDM-convexity of
2-separable convex function
\[
 f(x)=g(x)+ \sum_{i,j:j\not=i} \psi_{ij}(x_i+x_j)\qquad (x\in\ZZ^n).
\]

\begin{lemma}\label{lem:psi2ddmc}
For a univariate discrete convex function $\psi:\ZZ\to\RR\cup\{+\infty\}$,
$f:\ZZ^n\to\RR\cup\{+\infty\}$ defined by
\[
f(x)=\psi(x_1+x_2)\qquad (x\in\ZZ^n)
\]
is DDM-convex.
\end{lemma}
\begin{proof}
For every $x,y \in \ZZ^n$, we show that
\begin{equation}\label{eq1:psi2ddmc}
 \psi(x_1{+}x_2) + \psi(y_1{+}y_2) \geq 
   \psi(\dmid(x,y)_1{+}\dmid(x,y)_2) + \psi(\dmid(y,x)_1{+}\dmid(y,x)_2).
\end{equation}
Suppose that $\min\{x_1{+}x_2, y_1{+}y_2\}=x_1{+}x_2$ and 
$\max\{x_1{+}x_2, y_1{+}y_2\}=y_1{+}y_2$ without loss of generality.
By convexity of $\psi$, for every $a,b,p,q \in \ZZ$ such that
(i)~$a+b=p+q$, (ii)~$a \leq p \leq b$ and (iii)~$a \leq q \leq b$, we have
$\psi(a)+\psi(b)\geq \psi(p)+\psi(q)$.
Thus, it is enough to show that
\begin{align}
\label{cond1:psi2dmc}
&(x_1+x_2)+(y_1+y_2)=(\dmid(x, y)_1+\dmid(x, y)_2)+(\dmid(y, x)_1+\dmid(y, x)_2),\\
\label{cond2:psi2dmc}
&x_1+x_2\le \dmid(x, y)_1+\dmid(x, y)_2\le y_1+y_2,\\
\label{cond3:psi2dmc}
&x_1+x_2\le \dmid(y, x)_1+\dmid(y, x)_2\le y_1+y_2.
\end{align}
Obviously, (\ref{cond1:psi2dmc}) holds by
\[
x_1+y_1=\dmid(x, y)_1+\dmid(y, x)_1, \quad x_2+y_2=\dmid(x, y)_2+\dmid(y, x)_2.
\]
To show (\ref{cond2:psi2dmc}) and (\ref{cond3:psi2dmc}) under
$x_1+x_2 \leq y_1+y_2$, we consider
the following three cases separately:
Case 1: $x_1\leq y_1$ and $x_2\leq y_2$,
Case 2: $x_1>y_1$ and $x_2<y_2$,
Case 3: $x_1<y_1$ and $x_2>y_2$.

Case 1 ($x_1\leq y_1$ and $x_2\leq y_2$).
In this case, we have
$\dmid(x, y)_1=\lfloor\frac{x_1+y_1}{2}\rfloor$, 
$\dmid(x, y)_2=\lfloor\frac{x_2+y_2}{2}\rfloor$, 
$\dmid(y, x)_1=\lceil\frac{x_1+y_1}{2}\rceil$ and
$\dmid(y, x)_2=\lceil\frac{x_2+y_2}{2}\rceil$, which imply
\[
x_1\le \dmid(x, y)_1\le \dmid(y, x)_1\le y_1, \quad
x_2\le \dmid(x, y)_2\le \dmid(y, x)_2\le y_2.
\]
Conditions (\ref{cond2:psi2dmc}) and (\ref{cond3:psi2dmc}) are
direct consequences of the above inequalities.

Case 2 ($x_1>y_1$ and $x_2<y_2$).
In this case, under condition $x_1+x_2 \leq y_1+y_2$, $c_1=x_1-y_1$ and
$c_2=y_2-x_2$ satisfy $c_2\ge c_1>0$.
By the following calculations:
\begin{align*}
\dmid(x, y)_1=\left\lceil\frac{x_1+y_1}{2}\right\rceil
=x_1-\left\lfloor\frac{c_1}{2}\right\rfloor
=y_1+\left\lceil\frac{c_1}{2}\right\rceil, \\ 
\dmid(x, y)_2=\left\lfloor\frac{x_2+y_2}{2}\right\rfloor
=x_2+\left\lfloor\frac{c_2}{2}\right\rfloor
=y_2-\left\lceil\frac{c_2}{2}\right\rceil, \\
\dmid(y, x)_1=\left\lfloor\frac{x_1+y_1}{2}\right\rfloor
=x_1-\left\lceil\frac{c_1}{2}\right\rceil
=y_1+\left\lfloor\frac{c_1}{2}\right\rfloor, \\
\dmid(y, x)_2=\left\lceil\frac{x_2+y_2}{2}\right\rceil
=x_2+\left\lceil\frac{c_2}{2}\right\rceil
=y_2-\left\lfloor\frac{c_2}{2}\right\rfloor,
\end{align*}
we have
\begin{align*}
\dmid(x, y)_1{+}\dmid(x, y)_2
=x_1+x_2+
\left(\left\lfloor\frac{c_2}{2}\right\rfloor{-}\left\lfloor\frac{c_1}{2}\right\rfloor\right)
=y_1+y_2-
\left(\left\lceil\frac{c_2}{2}\right\rceil{-}\left\lceil\frac{c_1}{2}\right\rceil\right), \\
\dmid(y, x)_1{+}\dmid(y, x)_2
=x_1+x_2+
\left(\left\lceil\frac{c_2}{2}\right\rceil{-}\left\lceil\frac{c_1}{2}\right\rceil\right)
=y_1+y_2-
\left(\left\lfloor\frac{c_2}{2}\right\rfloor{-}\left\lfloor\frac{c_1}{2}\right\rfloor\right).
\end{align*}
Conditions (\ref{cond2:psi2dmc}) and (\ref{cond3:psi2dmc}) follow from 
$\left\lfloor\frac{c_2}{2}\right\rfloor-\left\lfloor\frac{c_1}{2}\right\rfloor \geq 0$ and
$\left\lceil\frac{c_2}{2}\right\rceil-\left\lceil\frac{c_1}{2}\right\rceil\geq 0$.

Case 3 ($x_1<y_1$ and $x_2>y_2$).
In this case, we can show (\ref{cond2:psi2dmc}) and (\ref{cond3:psi2dmc})
in the same way as Case~2.
\end{proof}
By Proposition~\ref{prop:basic-operations}~(4), Proposition~\ref{prop:LtoDDMC}
and Lemma~\ref{lem:psi2ddmc}, we have the next property.

\begin{theorem}\label{thm:2sep2ddmc}
Any 2-separable convex function is DDM-convex.
\end{theorem}
Hence, 2-separable convex function is integrally convex.

\section{Parallelogram inequality}\label{para-ineq}

Parallelogram inequality was originally proposed in \cite{MMTT-DMC} for
globally/locally discrete midpoint convex functions. 
By borrowing arguments from \cite{MMTT-DMC}, we show that
DDM-convex sets/functions have similar properties.

For every pair $(x, y)\in\ZZ^n\times \ZZ^n$ with $\|y-x\|_\infty=m$, 
we consider sets defined by
\begin{equation}\label{ddmcparalleloset}
A_k=\{i \mid y_i-x_i\ge k\}, \quad B_k=\{i \mid y_i-x_i\le -k\} \qquad(k=1, \dots,  m),
\end{equation}
for which
$A_1\supseteq A_2\supseteq\cdots \supseteq A_m$, 
$B_1\supseteq B_2\supseteq\cdots \supseteq B_m$, 
$A_1\cap B_1=\emptyset$ and $A_m\cup B_m\not=\emptyset$.

We first show the following property of DDM-convex sets.

\begin{theorem}\label{thm:parallelogramset}
Let $S\subseteq\ZZ^n$ be a DDM-convex set, $x, y\in S$ with 
$\|y-x\|_\infty=m$, and $J\subseteq\{1, 2, \dots, m\}$.
If $\{A_k\}$ and $\{B_k\}$ are defined by $(\ref{ddmcparalleloset})$ and
$d=\sum_{k\in J}(\vecone_{A_k}-\vecone_{B_k})$, we have 
$x+d\in S$ and $y-d\in S$.
\end{theorem}
To show this theorem, it is enough to verify
\begin{equation}\label{para-step0}
x+\sum_{k\in J}(\vecone_{A_k}-\vecone_{B_k})\in S 
\qquad (\forall J \subseteq \{1, 2, \dots, m\}).
\end{equation}
We first show that the decomposition 
$\sum_{k=1}^m (\vecone_{A_k} - \vecone_{B_k})$ of $y-x$ can be constructed
by using the operation $\dmid(\cdot,\cdot)$.

For every $x\in \ZZ^n$, let us consider multiset $D(x)$ of vectors by
the following recursive formula:
\begin{equation}\label{paralleloset1}
D(x)=
\begin{cases}
    \emptyset & (x=\veczero), \\
    \{x\} & (\|x\|_\infty=1), \\
    \{\dmid(x, \veczero), \dmid(\veczero, x)\} & (\|x\|_\infty=2), \\
    D(\dmid(x, \veczero))\cup D(\dmid(\veczero, x)) & (\|x\|_\infty\ge3),
\end{cases}
\end{equation}
where $\veczero$ denotes the $n$-dimensional zero vector. 
We give several propositions.

\begin{proposition}\label{prop:DecompoA}
If a multiset $\{d^k\in\{-1, 0, 1\}^n{\setminus}\{\veczero\} \mid k=1, \dots, m\}$ satisfies
\begin{equation}\label{icondition}
1\geq d^1_i \geq d^2_i \geq \cdots \geq d^m_i \geq 0 \;\mbox{ or }\;
 -1 \leq d^1_i \leq d^2_i \leq \cdots \leq d^m_i \leq 0  
\end{equation}
for each $ i\in\{1, \dots, n\}$, then
\begin{equation}\label{Dequal}
D\left(\sum_{k = 1}^m d^k\right)=\{d^k \mid k = 1,\ldots, m \}.
\end{equation}
\end{proposition}
\begin{proof}
We prove the assertion by induction on $m$.
The assertion obviously holds if $m \leq 1$.

Suppose that $m=2$.
By (\ref{icondition}), we have
\[
\|d^{1}{+}d^{2}\|_\infty = 2, \quad
\dmid(d^{1}{+}d^{2}, \veczero)=d^{1}, \quad 
\dmid(\veczero, d^{1}{+}d^{2})=d^{2},
\]
which, together with (\ref{paralleloset1}), imply the assertion.

Suppose that $m \geq 3$.
Let $K=\{1, 2, \dots, m\}$, 
$K^\mathrm{O}=\{k \in K \mid k \mbox{ is odd}\}$ and
$K^\mathrm{E}=\{k \in K \mid k \mbox{ is even}\}$.
By induction hypothesis together with $|K^\mathrm{O}|, |K^\mathrm{E}|<|K|$,
we obtain
\begin{equation}\label{oddevenD}
D\left(\sum_{k\in K^\mathrm{O}}d^k\right)=\{d^k \mid k\in K^\mathrm{O}\}, \quad
D\left(\sum_{k\in K^\mathrm{E}}d^k\right)=\{d^k \mid k\in K^\mathrm{E}\}.
\end{equation}
Furthermore, the claim below guarantees that
\begin{equation}\label{oddevenmu}
\mu\left(\sum_{k\in K}d^k, \veczero\right)
=\sum_{k\in K^\mathrm{O}}d^k, \quad
\mu\left(\veczero, \sum_{k\in K}d^k\right)
=\sum_{k\in K^\mathrm{E}}d^k.
\end{equation}
By combining (\ref{paralleloset1}), (\ref{oddevenD}) and (\ref{oddevenmu}), we have
\begin{align*}
D\left(\sum_{k\in K}d^k\right)
&=D\left(\mu\left(\sum_{k\in K}d^k, \veczero\right)\right)\cup D\left(\mu\left(\veczero, \sum_{k\in K}d^k\right)\right)\\
&=D\left(\sum_{k\in K^\mathrm{O}}d^k\right)\cup D\left(\sum_{k\in K^\mathrm{E}}d^k\right)\\
&=\{d^k \mid k\in K^\mathrm{O}\}\cup \{d^k \mid k\in K^\mathrm{E}\}\\
&=\{d^k \mid k\in K\}.
\end{align*}

\noindent \textbf{Claim:}
(i) $\mu\left(\sum_{k\in K}d^k, \veczero\right)=\sum_{k\in K^\mathrm{O}}d^k$, and 
(ii) $\mu\left(\veczero, \sum_{k\in K}d^k\right)=\sum_{k\in K^\mathrm{E}}d^k$.

\smallskip\noindent(Proof) 
We show (i) (and can show (ii) in the same way).
Let us fix $i \in \{1,\ldots,n\}$.
Assume that $\left(\sum_{k\in K}d^k\right)_i=l > 0$.
Since $d^{1}_i=\cdots=d^{l}_i=1$ and 
$d^{{l+1}}_i=\cdots=d^{m}_i=0$ by (\ref{icondition}),
we have 
\[
\mu\left(\sum_{k\in K}d^k, \veczero\right)_i
=\left\lceil\frac{l}{2}\right\rceil
=|\{1, \dots, l\} \cap \{k \mid k \mbox{ is odd}\}|
=\left(\sum_{k\in K^\mathrm{O}}d^k\right)_i.
\]
In the case where $\left(\sum_{k\in K}d^k\right)_i=-l < 0$, 
we have  $d^{1}_i=\cdots=d^{l}_i=-1$ and 
$d^{{l+1}}_i=\cdots=d^{m}_i=0$ by (\ref{icondition}),
and hence
\[
\mu\left(\sum_{k\in K}d^k, \veczero\right)_i
=\left\lfloor\frac{-l}{2}\right\rfloor
=-\left\lceil\frac{l}{2}\right\rceil
=-|\{1, \dots, l\}\cap \{k \mid k \mbox{ is odd}\}|
=\left(\sum_{k\in K^\mathrm{O}}d^k\right)_i.
\]
If $\left(\sum_{k\in K}d^k\right)_i=0$, by (\ref{icondition}), 
we have $d^{1}_i=\cdots=d^{m}_i=0$ and 
\[
\mu\left(\sum_{k\in K}d^k, \veczero\right)_i
=0
=\left(\sum_{k\in K^\mathrm{O}}d^k\right)_i.
\]
Thus, (i) holds.  (End of the proof of Claim).
\end{proof}

\begin{proposition}\label{prop:step2}
$D(y-x)=\{\vecone_{A_k}-\vecone_{B_k} \mid k=1, \dots, m\}$.
\end{proposition}
\begin{proof}
By the construction (\ref{ddmcparalleloset}), 
$y-x=\sum_{k=1}^m(\vecone_{A_k}-\vecone_{B_k})$.
By defining $d^k=\vecone_{A_k}-\vecone_{B_k}\ (k=1, \dots, m)$, 
condition (\ref{icondition}) of Proposition~\ref{prop:DecompoA} holds.
The assertion is an immediate consequence of (\ref{Dequal}).
\end{proof}

By Proposition~\ref{prop:step2}, (\ref{para-step0}) can be
rewritten as
\begin{equation}\label{step5-1}
x+\sum\{d \mid d\in E\}\in S \qquad (\forall E \subseteq D(y-x)).
\end{equation}
Therefore Theorem~\ref{thm:parallelogramset} can be shown by
the following proposition.

\begin{proposition}\label{prop:step5-1}
For a DDM-convex set $S$ and $x,y \in S$, $(\ref{step5-1})$ holds.
\end{proposition}
\begin{proof}
We prove (\ref{step5-1}) for $x,y$ by induction on $\|y-x\|_\infty$.
If $\|y-x\|_\infty \leq 1$, (\ref{step5-1}) trivially holds.
If $\|y-x\|_\infty=2$, then $D(y-x)=\{\dmid(y{-}x, \veczero), \ \dmid(\veczero, y{-}x)\}$.
Since $S$ is DDM-convex,  we have
\begin{align*}
x+\dmid(y{-}x, \veczero) &=\dmid(y, x)\in S, \\
x+\dmid(\veczero, y{-}x) &=\dmid(x, y)\in S,
\end{align*}
which guarantee that (\ref{step5-1}) holds.

Suppose that $\|y-x\|_\infty\ge3$, and (\ref{step5-1}) holds for every
$x'', y''\in S$ with $\|x''-y''\|_\infty<\|x-y\|_\infty$.
We fix $E\subseteq D(y-x)$ arbitrarily.
Let $x'=\dmid(x, y)$ and $y'=\dmid(y, x)$.
Then we have $y' - x = y - x' = \dmid(y-x,\veczero)$.
By DDM-convexity of $S$, $x'$ and $y'$ also belong to $S$.
By Proposition~\ref{prop:basic}~(4) and the assumption $\|y-x\|_\infty\ge3$, we have
\[
  \|y'-x\|_\infty = \|y-x'\|_\infty = 
  \|\dmid(y-x, \veczero)\|_\infty < \|y-x\|_\infty.
\]
By induction hypothesis, (\ref{step5-1}) holds for $(x, y')$ and $(x', y)$, 
and furthermore, by the equality $D(y'-x)=D(y-x')=D(\dmid(y-x, \veczero))$,
we have
\begin{align}\label{step5-2}
\begin{split}
&u=x+\sum\{d \mid d\in E\cap D(\dmid(y-x, \veczero))\}\in S, \\
&v=x'+\sum\{d \mid d\in E\cap D(\dmid(y-x, \veczero))\}\in S.
\end{split}
\end{align}
We also have $v-u = x'-x = \dmid(\veczero, y-x)$ and
\[
  \|v-u\|_\infty=\|x'-x\|_\infty=\|\dmid(\veczero, y-x)\|_\infty<\|y-x\|_\infty,
\]
which, together with the induction hypothesis,
guarantee that (\ref{step5-1}) holds for $(u,v)$.
Moreover, by (\ref{paralleloset1}), 
$D(y-x) = D(\dmid(y-x, \veczero))\cup D(\dmid(\veczero, y-x))$ includes $E$,
and hence, 
$D(v-u)=D(\dmid(\veczero, y-x))$ includes $E\setminus D(\dmid(y-x, \veczero))$.
Thus, we have
\begin{equation}\label{step5-3}
w=u+\sum\{d \mid d\in E\setminus D(\dmid(y-x, \veczero))\}\in S.
\end{equation}
By (\ref{step5-2}) and (\ref{step5-3}), we obtain
\[
w=x+\sum\{d \mid d\in E\}\in S,
\]
which implies (\ref{step5-1}).
\end{proof}

\medskip
We denote by DDMC($k$) and by DDMC($\geq$$k$) the classes of functions
$f : \ZZ^n \to \RR \cup \{+\infty\}$ that satisfy DDM-convexity (\ref{DDMC})
for all $x,y \in \ZZ^n$ with $\|x-y\|_\infty = k$ and 
for all $x,y \in \ZZ^n$ with $\|x-y\|_\infty \geq k$, respectively.
Before presenting parallelogram inequality for DDM-convex functions,
we give a useful property.

\begin{theorem}\label{thm:ddmcparalleloineq}
Let $f:\ZZ^n \to \RR \cup \{+\infty\}$ be a function in {\rm DDMC($2$)}
such that $\dom f$ is DDM-convex.
For $x\in \dom f$ and $y\in\ZZ^n$ with $\|y-x\|_\infty=m$, and for any partition
$(I,J)$ of $\{1, \dots, m\}$, we consider 
\[
d_1=\sum_{i\in I}(\vecone_{A_i}-\vecone_{B_i}), \quad
 d_2=\sum_{j\in J}(\vecone_{A_j}-\vecone_{B_j}),
\]
where $A_k, B_k\ (k=1, \dots, m)$ are the sets defined by 
$(\ref{ddmcparalleloset})$.
Then we have
\begin{equation}\label{ddmcparallelotheo1}
f(x)+f(x+d_1+d_2)\geq f(x+d_1)+f(x+d_2).
\end{equation}
\end{theorem}
\begin{proof}
We note that $y = x+d_1+d_2$.
If $y\not\in \dom f$, by $f(y)=+\infty$, (\ref{ddmcparallelotheo1})
trivially holds.
In the sequel, we assume that $y \in \dom f$.
Let $I$ be denoted by $\{i_1,i_2\ldots,i_{|I|}\}$ 
$(i_1 < i_2 < \cdots < i_{|I|})$
and $J$ by $\{j_1,j_2\ldots,j_{|J|}\}$ $(j_1 < j_2 < \cdots < j_{|J|})$.
For every $k = 1,\ldots,|I|$ and $l = 1,\ldots,|J|$, we denote
$d_1^k = \vecone_{A_{i_k}}-\vecone_{B_{i_k}}$ and, similarly,
$d_2^l = \vecone_{A_{j_l}}-\vecone_{B_{j_l}}$.
For every $k=0,1, \dots, |I|$ and for every $l=0,1, \dots, |J|$, define
\[
  x(k, l)=x+\sum_{i=1}^{k}d_1^ i + \sum_{j=1}^l d_2^j.
 \]
By Theorem~\ref{thm:parallelogramset}, for every $k, l$, we have 
$x(k, l) \in \dom f$.
We note that (\ref{ddmcparallelotheo1}) is equivalent to
\begin{equation}\label{DDMC-parallelo-proof2}
f(x(0, 0)) + f(x(|I|, |J|)) \geq f(x(|I|, 0)) + f(x(0, |J|)).
\end{equation}

 Fix $k \in\{1,\ldots,|I|\}$ and $l \in \{1,\ldots, |J|\}$.
According to whether $i_k > j_l$ or $i_k < j_l$, either
\[
  A_{i_k} \subseteq A_{j_l},\quad B_{i_k} \subseteq B_{j_l},\quad
   A_{j_l} \cap B_{j_l} = \emptyset,\quad A_{i_k} \cup B_{i_k} \neq \emptyset
\]
or
\[
  A_{j_l} \subseteq A_{i_k},\quad B_{j_l} \subseteq B_{i_k},\quad
  A_{i_k} \cap B_{i_k} = \emptyset,\quad A_{j_l} \cup A_{j_l} \neq \emptyset.
\]
Thus, in the case where $i_k > j_l$, we have
\[
(d_1^k+d_2^l)_p=
\begin{cases}
    1 & (p\in A_{j_l}\setminus A_{i_k}), \\
    2 & (p\in A_{i_k}), \\
    -2 & (p\in B_{i_k}), \\
    -1 & (p\in B_{j_l}\setminus B_{i_k}), \\
    0 & (p\not\in (A_{j_l}\cup B_{j_l})),
\end{cases}
\]
which implies $\|d_1^k+d_2^l\|_\infty = 2$, 
$\dmid(d_1^k+d_2^l, \veczero)=\vecone_{A_{j_l}}-\vecone_{B_{j_l}}=d^l_2$ and
$\dmid(\veczero, d_1^k+d_2^l)=\vecone_{A_{i_k}}-\vecone_{B_{i_k}}=d^k_1$.
Similarly, in the case where $i_k < j_l$, we have $\|d_1^k+d_2^l\|_\infty = 2$, 
$\dmid(d_1^k+d_2^l, \veczero)=d^k_1$, $\dmid(\veczero, d_1^k+d_2^l)=d^l_2$.
In both cases, since 
$\|x(k, l)-x(k-1, l-1)\|_\infty=\|d_1^k+d_2^l\|_\infty=2$ and $f \in $ DDMC($2$),
we obtain
\begin{align}\label{DDMC-parallelo-proof1}
\begin{split}
&f(x(k, l))+f(x(k-1, l-1)) \\
&\geq f(\dmid(x(k, l), x(k-1, l-1)))+f(\dmid(x(k-1, l-1), x(k, l))).
\end{split}
\end{align}
On the other hand, the facts
\begin{align*}
 \dmid(x(k, l), x(k-1, l-1)) &=x(k-1, l-1)+\dmid(d_1^k+d_2^l, \veczero), \\
 \dmid(x(k-1, l-1), x(k, l)) &=x(k-1, l-1)+\dmid(\veczero, d_1^k+d_2^l),
\end{align*}
together with 
$\{\dmid(d_1^k+d_2^l, \veczero),\dmid(\veczero, d_1^k+d_2^l)\} =
\{d_1^k, d_2^l\}$ in the both cases where $i_k > j_l$ and $i_k < j_l$, 
yield the right-hand side of 
(\ref{DDMC-parallelo-proof1}) is equal to $f(x(k, l-1))+f(x(k-1, l))$.
Therefore, we obtain
\[
 f(x(k, l)))+f(x(k-1, l-1)) \geq f(x(k, l-1))+f(x(k-1, l)).
\]

By adding the above inequalities for $(k,l)$ with $1\leq k\leq |I|$ and $1\leq l\leq |J|$, 
we obtain (\ref{DDMC-parallelo-proof2}).
We emphasize that all the terms that are canceled in this addition of inequalities are
finite valued because $x(k,l) \in \dom f$ for all $(k,l)$ with 
$0\leq k\leq |I|$ and $0\leq l\leq |J|$.
\end{proof}

The next theorem is an immediate consequence of Theorem~\ref{thm:ddmcparalleloineq},
because a DDM-convex function $f:\ZZ^n \to \RR \cup \{+\infty\}$ 
belongs to DDMC($2$) and $\dom f$ is DDM-convex.

\begin{theorem}\label{thm:ddmcparalleloineq2}
Let $f:\ZZ^n \to \RR \cup \{+\infty\}$ be a DDM-convex function.
For every $x, y\in \dom f$, let $\{A_k \mid k=1,\ldots,m\}$ and 
$\{B_k \mid k=1,\ldots,m\}$ be the families defined by 
$(\ref{ddmcparalleloset})$, where $m = \|y-x\|_\infty$.
For any subset $J \subseteq \{1,\ldots,m\}$, let
$d = \sum_{k\in J}(\vecone_{A_k}-\vecone_{B_k})$.
Then we have
\begin{equation}\label{parallelo-equ}
f(x)+f(y)\ge f(x+d)+f(y-d).
\end{equation}
 \end{theorem}
 We call the inequality (\ref{parallelo-equ}) {\em parallelogram inequality}
 of DDM-convex functions.

\section{Characterizations}\label{chars}

In this section, we give several equivalent conditions of DDM-convexity
and a simple characterization of quadratic DDM-convex functions.

For every pair $(x,y) \in \ZZ^n \times \ZZ^n$, we recall that
the families $\{A_k \mid k=1,\ldots,m\}$ and $\{B_k \mid k=1,\ldots,m\}$
are defined by
\[
A_k=\{i \mid y_i-x_i\ge k\}, \quad B_k=\{i \mid y_i-x_i\le -k\} 
\qquad(k=1, \dots,  m),
\]
in (\ref{ddmcparalleloset}), where $m = \|y-x\|_\infty$.

\begin{theorem}\label{thm:char}
For a function $f:\ZZ^n\to\RR\cup\{+\infty\}$, the following properties are equivalent
to each other.
\begin{enumerate}
\renewcommand{\labelenumi}{\rm (\arabic{enumi})}
\item 
$f$ is DDM-convex, i.e., $f\in$ {\rm DDMC($\geq$$1$)}.
\item 
$\dom f$ is DDM-convex and $f\in$ {\rm DDMC($2$)}.
\item 
$f$ satisfies parallelogram inequalities 
for every $x, y\in \dom f$ and for any subset 
$J \subseteq \{1,\ldots,m\}$ where $m = \|y-x\|_\infty$.
\item 
For every $x, y\in \dom f$, we have
\begin{equation}\label{DAPR}
f(x)+f(y)\ge f(x+\vecone_{A_m}-\vecone_{B_m})+f(y-\vecone_{A_m}+\vecone_{B_m}),
\end{equation}
where $m = \|y-x\|_\infty$, and the sets $A_m$ and $B_m$ are
defined by $(\ref{ddmcparalleloset})$.
\item
For every $x, y\in \dom f$ and for any $\alpha \in \{1, \dots, m\}$,
by defining $d=\sum_{k=1}^\alpha (\vecone_{A_k}-\vecone_{B_k})$,
we have
\[
  f(x)+f(y)\ge f(x+d)+f(y-d),
\]
where $m = \|y-x\|_\infty$, and the families $\{A_k \mid k=1,\ldots,m\}$ 
and $\{B_k \mid k=1,\ldots,m\}$ are defined by 
$(\ref{ddmcparalleloset})$.
\end{enumerate}
\end{theorem}
\begin{proof}
The implication $(1) \Rightarrow (2)$ is obvious and 
the implications from (1) to (3), (4) and (5) follow from 
Theorem~\ref{thm:ddmcparalleloineq2}.
To prove opposite implications, we show $f \in$ DDMC($m$) for every $m\geq 1$
from (4), and show [$(5) \Rightarrow (4)$], [$(3) \Rightarrow (4)$] and
[$(2) \Rightarrow (3)$].

(4)$\Rightarrow$(1):
Since $f \in$ DDMC($1$) always holds from Proposition~\ref{prop:basic}, we assume
$m\geq 2$ and fix $x,y \in \dom f$ with $\|x-y\|_\infty=m$, arbitrarily.
Let denote $p=x+\vecone_{A_m}-\vecone_{B_m}$ and 
$q=y-\vecone_{A_m}+\vecone_{B_m}$ in (\ref{DAPR}).
Vectors $p$ and $q$ satisfy conditions (a) and (c) of Proposition~\ref{prop:dmid},
$\|p-q\|_\infty<m$ and $p,q \in \dom f$.
By repeating (\ref{DAPR}) for $(p,q)$ until $\|p-q\|_\infty \leq 1$,
the final $p$ and $q$ satisfy all conditions of  Proposition~\ref{prop:dmid}, and hence,
$p=\dmid(x, y)$ and $q=\dmid(y, x)$ are satisfied.
This shows $f \in$ DDMC($m$) holds.

(5)$\Rightarrow$(4):
 By setting $\alpha = m-1$ in (5), we obtain 
\[
 x+d=y-\vecone_{A_m}+\vecone_{B_m}, \quad y-d=x+\vecone_{A_m}-\vecone_{B_m},
\]
and hence, (4).

(3)$\Rightarrow$(4):
By setting $J=\{1, \dots, m-1\}$ in (3), we obtain
\[
 x+d=y-\vecone_{A_m}+\vecone_{B_m}, \quad y-d=x+\vecone_{A_m}-\vecone_{B_m},
\]
and hence, (4).

(2)$\Rightarrow$(3): Property (2) and Theorem~\ref{thm:ddmcparalleloineq}
imply (3).
\end{proof}

\begin{remark}
Theorem~\ref{thm:char} may be regarded as a generalization of 
Theorem~\ref{thm:L-char}.

Property (4) of Theorem~\ref{thm:char} corresponds to 
(4) of Theorem~\ref{thm:L-char}.
For an L$^\natural$-convex function $f$, the following inequalities 
\[
  f(x) + f(y) \geq f(x + \vecone_{A_m}) + f(y - \vecone_{A_m})
\]
and
\[
  f(x) + f(y) \geq f(x - \vecone_{B_m}) + f(y + \vecone_{B_m})
\]
hold.
These inequalities imply (\ref{DAPR}).
However, for a DDM-convex function $f$, the two parts on $A_m$ and 
$B_m$ must be combined to a single inequality (\ref{DAPR}).
This demonstrates the difference between L$^\natural$-convexity and
DDM-convexity.

Property (5) of Theorem~\ref{thm:char} corresponds to 
(1) of Theorem~\ref{thm:L-char}.
In (5) of Theorem~\ref{thm:char}, two vectors 
$x+d=x+\sum_{k=1}^\alpha (\vecone_{A_k}-\vecone_{B_k})$ and
$y-d=y-\sum_{k=1}^\alpha (\vecone_{A_k}-\vecone_{B_k})$ are 
expressed by
\[
(x+d)_i=
\begin{cases}
x_i-\alpha & (x_i-y_i\ge \alpha), \\
x_i+\alpha & (y_i-x_i\ge \alpha), \\
y_i & (|x_i-y_i|<\alpha),
\end{cases}
 \quad
(y-d)_i=
\begin{cases}
y_i+\alpha & (x_i-y_i\ge \alpha), \\
y_i-\alpha & (y_i-x_i\ge \alpha), \\
x_i & (|x_i-y_i|< \alpha).
\end{cases}
\]
On the other hand, in (\ref{trans-submo}),
two vectors $(x-\alpha\vecone)\vee y$ and $x\wedge(y+\alpha\vecone)$
can be rewritten as
\begin{align*}
\left((x-\alpha\vecone)\vee y\right)_i &=
\begin{cases}
x_i-\alpha & (x_i-y_i\ge \alpha), \\
y_i & (x_i-y_i<\alpha),
\end{cases} \\
\left(x\wedge (y+\alpha\vecone) \right)_i &=
\begin{cases}
y_i+\alpha & (x_i-y_i\ge \alpha), \\
x_i & (x_i-y_i<\alpha).
\end{cases}
\end{align*}
In the same way as the relation between (4) of Theorem~\ref{thm:char} and 
(4) of Theorem~\ref{thm:L-char}, two separate operations for 
L$^\natural$-convex functions must be executed simultaneously for
DDM-convex functions.
\hfill\finbox
\end{remark}

For a quadratic function $f(x) = x^\top Q x\;(x\in \ZZ^n)$ with a symmetric matrix
$Q = [q_{ij}]$, we show that $f$ is DDM-convex if and only if $Q$ is 
{\em diagonally dominant with nonnegative diagonals}:
\begin{equation}\label{diagdom}
 q_{ii} \geq \sum_{j:j \neq i}|q_{ij}| \qquad (\forall i =1,\ldots,n).
\end{equation}
For each $p\in\RR$, let $p^+=\max\{p, 0\}$ and $p^-=\max\{-p, 0\}$.
Note that $|p|=p^{+} + p^{-}$.
Quadratic function $f(x)=x^\top Q x$ can be written as\footnote{
Private communication with F. Tardella (2017).}
\begin{align}
f(x) &=\sum_{i, j=1}^n q_{ij}x_ix_j \notag \\
 &=\sum_{i=1}^n\left(q_{ii}{-}\sum_{j:j\not=i}|q_{ij}|\right)x_i^2
   +\frac{1}{2}\sum_{i,j:j\not=i}(q_{ij}^+(x_i{+}x_j)^2+q_{ij}^-(x_i{-}x_j)^2).
      \label{quadratic2sep}
\end{align}
In (\ref{quadratic2sep}), the condition (\ref{diagdom}) of $Q$ implies
the nonnegativity of coefficients of $x_i^2$.
Thus, if $Q$ is diagonally dominant with nonnegative diagonals, 
then $f$ is 2-separable convex, and hence,
DDM-convex by Theorem~\ref{thm:2sep2ddmc}.
By proving the opposite implication, we obtain the following property.

\begin{theorem}\label{thm:diagdom-DDMC}
For a quadratic function $f(x) = x^\top Q x\; (x \in \ZZ^n)$ 
with a symmetric matrix $Q = [q_{ij}]$, $f$ is DDM-convex if and only if
$Q$ is diagonally dominant with nonnegative diagonals.
\end{theorem}
\begin{proof}
It is enough to show that if $f$ is DDM-convex, then $Q$ is 
diagonally dominant with nonnegative diagonals.
For each $ i\in\{1, \dots, n\}$, define $z^i\in\ZZ^n$ by
\[
z^i_j=
\begin{cases}
    2 & \  (j=i), \\
    -1  & \ (j\not=i \mbox{ and } q_{ij}\ge0), \\
    1 & \  (j\not=i \mbox{ and } q_{ij}<0).
\end{cases}
\]
By DDM-convexity of $f$, the inequality
\[
f(z^i)+f(\veczero)\ge f(\dmid(z^i, \veczero))+f(\dmid(\veczero, z^i))
\]
must hold.
Since $\dmid(z^i, \veczero)=z^i-\vecone_i$ and
$\dmid(\veczero, z^i)=\vecone_i$, we have
\[
  (z^i)^\top Q z^i \geq (z^i-\vecone_i)^\top Q(z^i-\vecone_i)+\vecone_i^\top Q\vecone_i,
\]
which implies
\begin{align*}
0 &\leq (z^i)^\top Q \vecone_i -\vecone_i^\top Q \vecone_i \\
 &=q_{ii}-\sum_{j:j\not=i\wedge q_{ji}\ge0} q_{ji}+\sum_{j:j\not=i\wedge q_{ji}<0} q_{ji} \\
 &=q_{ii}-\sum_{j:j\not=i} |q_{ji}|.
\end{align*}
By $Q^\top = Q$, we obtain the diagonal dominance 
with nonnegative diagonals of $Q$.
\end{proof}

The minimizers of DDM-convex functions are DDM-convex sets, while
the minimizers of L$^\natural$-convex functions are L$^\natural$-convex sets.
The class of L$^\natural$-convex functions has a characterization
in terms of minimizers.
For a function $f : \ZZ^n \to \RR \cup \{+\infty\}$ and $p \in \RR^n$, we
denote by $f-p$ the function given by
\[
 (f-p)(x) = f(x) - \sum_{i=1}^n p_ix_i \qquad (\forall x \in \ZZ^n).
\]

\begin{theorem}[\cite{DCA,MS00polyML}]
For a function $f : \ZZ^n \to \RR \cup \{+\infty\}$, $f$ is L$^\natural$-convex
if and only if $\argmin (f-p)$ is an L$^\natural$-convex set for every $p \in \RR^n$
with $\inf (f-p) >-\infty$.
\end{theorem}

Unfortunately, the class of DDM-convex functions does not have a similar characterization.

\begin{example}
Let us consider the function $f : \ZZ^3 \to \RR \cup \{+\infty\}$ given by
\begin{align*}
 & f(0,0,0) = 0,\quad f(1,0,1) = f(1,1,0) = 1, \\
 & f(1,0,0) = f(1,1,1) = 2,\quad f(2,1,1) = 3, \\
 & f(x) = +\infty \quad(\mbox{otherwise}).
\end{align*}
The function $f$ is not DDM-convex because
\[
 f(0,0,0) + f(2,1,1) = 3 < 4 = f(1,0,0) + f(1,1,1),
\]
while $\dom f$ is a DDM-convex set (in fact, an L$^\natural$-convex set).
Futhermore, $\argmin (f-p)$ is a DDM-convex set for every $p \in \RR^3$ as follows.
There exists no $p \in \RR^3$ such that $\{(0,0,0),(2,1,1)\} \subseteq \argmin(f-p)$,
because we have $0 \leq 1-p_1-p_2$ from $(f-p)(0,0,0) \leq (f-p)(1,1,0)$, and
$2-p_1-p_2 \leq 0$ from $(f-p)(2,1,1) \leq (f-p)(1,0,1)$.
For any $p \in \RR^3$ and for any $x,y \in \argmin(f-p)$, this fact implies
that $\|x-y\|_\infty \leq 1$ must hold, and hence, $\argmin(f-p)$ is
a DDM-convex set. 
We note that this example also shows that a similar characterization does not hold
for the classes of globally/locally discrete midpoint convex functions.
\hfill \finbox
\end{example}

\section{Operations}\label{operations}

We discuss several operations for discrete convex functions, including 
scaling operations~\cite{MMTT2019,MMTT-DMC,DCA},  
restriction, projection, direct sum and 
convolution operations~\cite{projectionandconvolution,DCA,DCAOperation}.

\subsection{Scaling operations}

Scaling operations are useful techniques for designing efficient algorithms in 
discrete optimization.
It is shown in \cite{MMTT-DMC} that 
global/local discrete midpoint convexity,
including L$^\natural$-convexity, is closed under scaling operations.
We show that DDM-convexity is also closed under scaling operations.

Given a function $f:\ZZ^n\to\RR\cup\{+\infty\}$ and 
a positive integer $\alpha$, the {\em $\alpha$-scaling} of $f$
is the function $f^\alpha$ defined by
\[
  f^\alpha(x)=f(\alpha x)\qquad(x\in\ZZ^n).
\]
We also define the $\alpha$-scaling $S^\alpha$ of a set $S \subseteq \ZZ^n$ by
\[
  S^\alpha=\{x\in\ZZ^n \mid \alpha x\in S\}.
\]

\begin{theorem}\label{thm:scaling}
Given a DDM-convex function $f:\ZZ^n \to \RR \cup \{+\infty\}$
and a positive integer $\alpha$, the scaled function 
$f^\alpha:\ZZ^n\to\RR\cup\{+\infty\}$ is also DDM-convex.
\end{theorem}
\begin{proof}
By the equivalence between (1) and (4) of Theorem~\ref{thm:char},
it is sufficient to show that
\[
f^\alpha(x)+f^\alpha(y) \geq 
f^\alpha(x+\vecone_{A_m}-\vecone_{B_m})+f^\alpha(y-\vecone_{A_m}+\vecone_{B_m})
\]
for every $x, y\in\ZZ^n$ with $\|x-y\|_\infty = m$ 
and for families $\{A_k \mid k=1,\ldots,m\}$ and 
$\{B_k \mid k=1,\ldots,m\}$ defined by (\ref{ddmcparalleloset}).
The above inequality is written as
\begin{equation}\label{scal}
f(\alpha x)+f(\alpha y) \geq 
f(\alpha(x+\vecone_{A_m}-\vecone_{B_m}))+f(\alpha(y-\vecone_{A_m}+\vecone_{B_m})).
\end{equation}
For $\alpha x$ and $\alpha y$ with $\|\alpha x-\alpha y\|_\infty=\alpha m$, by defining
\[
A_l^\alpha=\{i \mid \alpha y_i-\alpha x_i\ge l\}, \quad 
B_l^\alpha=\{i \mid \alpha y_i-\alpha x_i\le -l\}\qquad(l=1, \dots, \alpha m),
\]
we have 
$\alpha y-\alpha x=\sum_{l=1}^{\alpha m}(\vecone_{A_l^\alpha}-\vecone_{B_l^\alpha})$ 
and
\[
A_k=A_{\alpha (k-1)+j}^\alpha,  \quad 
B_k=B_{\alpha (k-1)+j}^\alpha \qquad (k=1, \dots, m;\;  j=1, \dots, \alpha).
\]
By (5) of Theorem~\ref{thm:char} for
\[
d =\sum_{l=1}^{\alpha (m-1)} (\vecone_{A^\alpha_l}-\vecone_{B^\alpha_l}), \quad 
d' =\sum_{l=\alpha (m-1)+1}^{\alpha m} (\vecone_{A^\alpha_l}-
  \vecone_{B^\alpha_l})=\alpha(\vecone_{A_m}-\vecone_{B_m})
\]
we have
\begin{align*}
f(\alpha x)+f(\alpha y)&\ge f(\alpha x+d)+f(\alpha y-d)\\
&=f(\alpha y-d')+f(\alpha x+d')\\
&=f(\alpha(y-\vecone_{A_m}+\vecone_{B_m}))
 +f(\alpha(x+\vecone_{A_m}-\vecone_{B_m})),
\end{align*}
that is, (\ref{scal}).
\end{proof}

\begin{corollary}\label{prop:scalingset}
For a DDM-convex set $S\subseteq\ZZ^n$ and a positive integer $\alpha$,
the $\alpha$-scaled set $S^\alpha$ is also DDM-convex.
\end{corollary}

\subsection{Restrictions}

For a function $f:\ZZ^{n+m}\to\RR\cup\{+\infty\}$, the {\em restriction} of $f$
on $\ZZ^n$ is the function $g$ defined by
\[
   g(x)=f(x, \veczero) \qquad (x\in\ZZ^n).
\]
For a set $S\subseteq\ZZ^{n+m}$, the restriction of $S$ on $\ZZ^n$ is also defined by
\[
  T=\{x\in\ZZ^n \mid (x, \veczero)\in S\}.
\]
Obviously, the following properties hold.

\begin{proposition}\label{prop:restriction}
For a DDM-convex function, its restrictions are also DDM-convex. 
\end{proposition}

\begin{proposition}
For a DDM-convex set, its restrictions are also DDM-convex.
\end{proposition}

\subsection{Projections}

For a function $f:\ZZ^{n+m}\to\RR\cup\{+\infty\}$, the {\em projection} of $f$
to $\ZZ^n$ is the function defined by
\begin{equation}\label{def:projectionfn}
g(x)=\inf\{f(x, y) \mid y\in\ZZ^{m}\}\qquad(x\in \ZZ^n),
\end{equation}
where we assume that $g(x)>-\infty$ for all $x \in \ZZ^n$.
For a set $S\subseteq\ZZ^{n+m}$, the projection of $S$ to $\ZZ^n$ is also defined by
\begin{equation*}\label{def:projectionset}
T=\{x\in\ZZ^n \mid \exists y\in\ZZ^m : (x, y)\in S\}.
\end{equation*}
In the same way as the proof for 
globally discrete midpoint convex functions
in \cite[Theorem 3.5]{projectionandconvolution} we can show the following property.

\begin{proposition}
For a DDM-convex function, its projections are DDM-convex.
\end{proposition}
\begin{proof}
Let $g$ be the projection defined by (\ref{def:projectionfn}) of 
a DDM-convex function $f$.
For every $x^{(1)}, x^{(2)} \in \dom g$ and every $\varepsilon>0$, 
by the definition of the projection, there exist $y^{(1)}, y^{(2)}\in\ZZ^m$ with
$g(x^{(i)})\geq f(x^{(i)}, y^{(i)})-\varepsilon$ for $i=1, 2$.
Thus, we have
\begin{equation}\label{ddmcprojectionproof1}
g(x^{(1)})+g(x^{(2)})\ge f(x^{(1)}, y^{(1)})+f(x^{(2)}, y^{(2)})-2\varepsilon.
\end{equation}
By DDM-convexity of $f$ and the definition of the projection, we have
\begin{align}
& f(x^{(1)}, y^{(1)})+f(x^{(2)}, y^{(2)}) \notag \\
& \ge f(\dmid((x^{(1)}, y^{(1)}), (x^{(2)}, y^{(2)})))+
   f(\dmid((x^{(2)}, y^{(2)}), (x^{(1)}, y^{(1)}))) \notag \\ 
&\ge g(\dmid(x^{(1)}, x^{(2)}))+g(\dmid(x^{(2)}, x^{(1)})). \label{ddmcprojectionproof2}
\end{align}
By (\ref{ddmcprojectionproof1}) and (\ref{ddmcprojectionproof2}),
we obtain
\[
g(x^{(1)})+g(x^{(2)})\ge g(\dmid(x^{(1)}, x^{(2)}))+g(\dmid(x^{(2)}, x^{(1)}))-2\varepsilon
\]
for any $\varepsilon>0$, which guarantees DDM-convexity of $g$.
\end{proof}

\begin{corollary}
For a DDM-convex set, its projections are also DDM-convex.
\end{corollary}

\subsection{Direct sums}

For two functions $f_1:\ZZ^{n_1}\to\RR\cup\{+\infty\}$ and 
$f_2:\ZZ^{n_2}\to\RR\cup\{+\infty\}$, the {\em direct sum} of $f_1$ and $f_2$
is the function $f_1\oplus f_2:\ZZ^{n_1+n_2}\to\RR\cup\{+\infty\}$ defined by
\[
(f_1\oplus f_2)(x, y)=f_1(x)+f_2(y) \qquad(x\in\ZZ^{n_1},\; y\in\ZZ^{n_2}).
\]
For two sets $S_1\subseteq\ZZ^{n_1}$ and $S_2\subseteq\ZZ^{n_2}$, the direct sum
of $S_1$ and $S_2$ is defined by
\[
S_1\oplus S_2=\{(x, y) \mid x\in S_1,\; y\in S_2\}.
\]
DDM-convexity is closed under direct sums as below.

\begin{proposition}
For two DDM-convex functions $f_1:\ZZ^{n_1}\to\RR\cup\{+\infty\}$ and
$f_2:\ZZ^{n_2}\to\RR\cup\{+\infty\}$, the direct sum $f_1\oplus f_2$ 
is also DDM-convex.
\end{proposition}
\begin{proof}
For every $x^{(1)}, y^{(1)}\in\ZZ^{n_1}$ and $x^{(2)}, y^{(2)}\in\ZZ^{n_2}$,
it follows from DDM-convexity of $f_1$ and $f_2$ that
\begin{align}
&f_1(x^{(1)})+f_1(y^{(1)})\geq f_1(\dmid(x^{(1)}, y^{(1)}))+f_1(\dmid(y^{(1)}, x^{(1)})),
\label{1def}\\
&f_2(x^{(2)})+f_2(y^{(2)})\geq f_2(\dmid(x^{(2)}, y^{(2)}))+f_2(\dmid(y^{(2)}, x^{(2)})).
\label{2def}
\end{align}
By the following relations
\begin{align*}
&\dmid((x^{(1)}, x^{(2)}), (y^{(1)}, y^{(2)}))
  =(\dmid(x^{(1)}, y^{(1)}), \dmid(x^{(2)}, y^{(2)})), \\ 
&\dmid((y^{(1)}, y^{(2)}), (x^{(1)}, x^{(2)}))
  =(\dmid(y^{(1)}, x^{(1)}), \dmid(y^{(2)}, x^{(2)})),
\end{align*}
and by (\ref{1def}) and (\ref{2def}), we obtain
\begin{align*}
&(f_1\oplus f_2)(x^{(1)}, x^{(2)})+(f_1\oplus f_2)(y^{(1)}, y^{(2)})\\
&=[f_1(x^{(1)})+f_2(x^{(2)})]+[f_1(y^{(1)})+f_2(y^{(2)})]\\
&\ge [f_1(\dmid(x^{(1)}, y^{(1)}))+f_2(\dmid(x^{(2)}, y^{(2)}))]+[f_1(\dmid(y^{(1)}, x^{(1)}))+f_2(\dmid(y^{(2)}, x^{(2)}))]\\
&=(f_1\oplus f_2)( (\dmid(x^{(1)}, y^{(1)}), \dmid(x^{(2)}, y^{(2)})) )+
(f_1\oplus f_2)( (\dmid(y^{(1)}, x^{(1)}), \dmid(y^{(2)}, x^{(2)})) )\\
&=(f_1\oplus f_2)( \dmid((x^{(1)}, x^{(2)}), (y^{(1)}, y^{(2)})) )+
(f_1\oplus f_2)( \dmid((y^{(1)}, y^{(2)}), (x^{(1)}, x^{(2)})) ),
\end{align*}
which says DDM-convexity of $f_1 \oplus f_2$.
\end{proof}

\begin{corollary}
For two DDM-convex sets $S_1\subseteq\ZZ^{n_1}$ and $S_2\subseteq\ZZ^{n_2}$,
$S_1\oplus S_2$ is also DDM-convex.
\end{corollary}

\subsection{Convolutions}

For two functions $f_1, f_2:\ZZ^{n}\to\RR\cup\{+\infty\}$, the {\em convolution}
$f_1\square f_2$ is the function defined by
\[
(f_1\square f_2)(x)=\inf\{f_1(y)+f_2(z) \mid  x=y+z,\; \ y, z\in\ZZ^n\} \qquad (x\in\ZZ^n),
\]
where we assume  $(f_1\square f_2)(x)>-\infty$ for every $x \in \ZZ^n$.
For two sets $S_1, S_2\subseteq\ZZ^n$, the {\em Minkowski sum} $S_1 + S_2$
defined by
\[
S_1+S_2=\{x+y \mid x\in S_1, y\in S_2 \}
\]
corresponds to the convolution of indicator functions 
$\delta_{S_1}$ and $\delta_{S_2}$.
The next example shows that the Minkowski sum of two DDM-convex sets
may not be DDM-convex, and hence, DDM-convexity is not closed under
the convolutions.

\begin{example}
We borrow the example in \cite[Example 4.2]{projectionandconvolution}, which shows
that L$^\natural$-convexity and global/local discrete midpoint convexity
may not be closed under the convolutions.

 Let $S_1=\{(0, 0, 0), (1, 1, 0)\}$ and $S_2=\{(0, 0, 0), (0, 1, 1)\}$ which are
 DDM-convex.
 The Minkowski sum of $S_1$ and $S_2$
\[
S_1+S_2=\{(0, 0, 0), (0, 1, 1), (1, 1, 0), (1, 2, 1)\}
\]
is not DDM-convex, because for $x=(0, 0, 0)$ and $y=(1, 2, 1)$,
we have $\dmid(x, y)=(0, 1, 0)\not\in S_1+S_2$ and
$\dmid(y, x)=(1, 1, 1)\not\in S_1+S_2$.
\hfill\finbox
\end{example}

It is known that the convolution of an L$^\natural$-convex function and 
a separable convex function is also L$^\natural$-convex, where
a separable convex function $\varphi : \ZZ^n \to \RR \cup \{+\infty\}$ is given by
\[
 \varphi(x) = \sum_{i=1}^n \varphi_i(x_i) \qquad (x \in \ZZ^n)
\]
for univariate discrete convex functions $\varphi_i (i=1,\ldots,n)$.
By the same arguments of \cite[Proposition 4.7]{projectionandconvolution}, 
this can be extended to DDM-convexity.

\begin{proposition}
The convolution of a DDM-convex function and a separable convex function 
is also DDM-convex.
\end{proposition}
\begin{proof}
Let $f :\ZZ^n\to\RR\cup\{+\infty\}$ be a DDM-convex function,
$\varphi:\ZZ^n\to\RR\cup\{+\infty\}$ a separable convex function 
represented as $\sum_{i=1}^n \varphi_i$ and let $g=f\square \varphi$.
For every $x^{(1)}, x^{(2)}\in\dom  g$ and $\varepsilon>0$, by the definition of
convolutions, there exist $y^{(i)}, z^{(i)}\; (i=1, 2)$ such that
\begin{equation}\label{ddmc-sepa-conv-ineq1}
g(x^{(i)})\geq f(y^{(i)})+\varphi(z^{(i)})-\varepsilon, \ \ x^{(i)}=y^{(i)}+z^{(i)}
\qquad (i=1, 2).
\end{equation}
It follows from DDM-convexity of $f$ that
\begin{equation}\label{ddmc-sepa-conv-ineq2}
f(y^{(1)})+f(y^{(2)})\geq f(\dmid(y^{(1)}, y^{(2)}))+f(\dmid(y^{(2)}, y^{(1)})).
\end{equation}
Let 
\[
z^\prime=\dmid(x^{(1)}, x^{(2)})-\dmid(y^{(1)}, y^{(2)}), \quad 
z^{\prime\prime}=\dmid(x^{(2)}, x^{(1)})-\dmid(y^{(2)}, y^{(1)}).
\]
By $g = f\square\varphi$, we have
\begin{align}
f(\dmid(y^{(1)}, y^{(2)}))+\varphi(z^\prime)\geq g(\dmid(x^{(1)}, x^{(2)})), 
  \label{ddmc-sepa-conv-ineq3} \\
f(\dmid(y^{(2)}, y^{(1)}))+\varphi(z^{\prime\prime})\geq g(\dmid(x^{(2)}, x^{(1)})).
  \label{ddmc-sepa-conv-ineq4}
\end{align}
The claim, below, states that
\begin{equation}\label{ddmc-sepa-conv-cl}
\varphi_i(z_i^{(1)})+\varphi_i(z_i^{(2)})\ge \varphi_i(z_i^\prime)+\varphi_i(z_i^{\prime\prime})
\end{equation}
for each $i=1, \dots, n$.
Thus, we have
\begin{equation}\label{ddmc-sepa-conv-ineq5}
\varphi(z^{(1)})+\varphi(z^{(2)})\ge \varphi(z^{\prime})+\varphi(z^{\prime\prime}).
\end{equation}
By summing up (\ref{ddmc-sepa-conv-ineq1}), (\ref{ddmc-sepa-conv-ineq2}), 
(\ref{ddmc-sepa-conv-ineq3}), (\ref{ddmc-sepa-conv-ineq4}) and
(\ref{ddmc-sepa-conv-ineq5}), we obtain
\[
g(x^{(1)})+g(x^{(2)})\ge g(\dmid(x^{(1)}, x^{(2)}))+g(\dmid(x^{(2)}, x^{(1)}))-2\varepsilon
\]
for any $\varepsilon>0$, which guarantees DDM-convexity of $g$.

\smallskip
\noindent \textbf{Claim:} (\ref{ddmc-sepa-conv-cl}) holds.

\smallskip\noindent(Proof) 
Since $\varphi_i$ is univariate discrete convex, 
for every $a, b\in\ZZ$ with $a\le b$ and for every $p, q\in\ZZ$ such that
(i) $a+b=p+q$, (ii) $a \leq p \leq b$ and (iii) $a \leq  q \leq b$, we have
$\varphi_i(a)+\varphi_i(b) \geq \varphi_i(p)+\varphi_i(q)$.
Thus, it is enough to show the following relations:
\begin{align}
&z_i^{(1)}+z_i^{(2)}=z_i^\prime+z_i^{\prime\prime}, 
  \label{ddmc-sepa-conv-cl-ineq1} \\
&\min\{z_i^{(1)}, z_i^{(2)}\} \leq z_i^\prime \leq \max\{z_i^{(1)}, z_i^{(2)}\}, 
  \label{ddmc-sepa-conv-cl-ineq2} \\
&\min\{z_i^{(1)}, z_i^{(2)}\} \leq z_i^{\prime\prime} \leq \max\{z_i^{(1)}, z_i^{(2)}\}.
  \label{ddmc-sepa-conv-cl-ineq3}
\end{align}
Condition (\ref{ddmc-sepa-conv-cl-ineq1}) follows from
\begin{align*}
  \dmid(x^{(1)}, x^{(2)})_i + \dmid(x^{(2)}, x^{(1)})_i &= x_i^{(1)} + x_i^{(2)},  \\
  \dmid(y^{(1)}, y^{(2)})_i + \dmid(y^{(2)}, y^{(1)})_i &= y_i^{(1)} + y_i^{(2)},
\end{align*}
and
\begin{align*}
z_i^\prime+z_i^{\prime\prime}
&=(\dmid(x^{(1)}, x^{(2)})_i-\dmid(y^{(1)}, y^{(2)})_i )
+( \dmid(x^{(2)}, x^{(1)})_i-\dmid(y^{(2)}, y^{(1)})_i )\\
&=( \dmid(x^{(1)}, x^{(2)})_i+\dmid(x^{(2)}, x^{(1)})_i  )
-( \dmid(y^{(1)}, y^{(2)})_i+\dmid(y^{(2)}, y^{(1)})_i )\\
&=(x_i^{(1)}+x_i^{(2)})-(y_i^{(1)}+y_i^{(2)})
=(x_i^{(1)}-y_i^{(1)})+(x_i^{(2)}-y_i^{(2)}) \\
&=z_i^{(1)}+z_i^{(2)}.
\end{align*}
To show (\ref{ddmc-sepa-conv-cl-ineq2}) and (\ref{ddmc-sepa-conv-cl-ineq3}),
we consider the following cases:
Case 1: $x^{(1)}_i\geq x^{(2)}_i$ and $y^{(1)}_i\geq y^{(2)}_i$,
Case 2: $x^{(1)}_i< x^{(2)}_i$ and $y^{(1)}_i< y^{(2)}_i$,
Case 3: $x^{(1)}_i\ge x^{(2)}_i$ and $y^{(1)}_i< y^{(2)}_i$,
Case 4: $x^{(1)}_i< x^{(2)}_i$ and $y^{(1)}_i\ge y^{(2)}_i$.

Case 1 ($x^{(1)}_i\geq x^{(2)}_i$ and $y^{(1)}_i\geq y^{(2)}_i$).
In this case, we have
\[
z^\prime_i=\left\lceil \frac{x^{(1)}_i{+}x^{(2)}_i}{2} \right\rceil - 
 \left\lceil \frac{y^{(1)}_i{+}y^{(2)}_i}{2} \right\rceil , \quad
z^{\prime\prime}_i=\left\lfloor \frac{x^{(1)}_i{+}x^{(2)}_i}{2} \right\rfloor - 
 \left\lfloor \frac{y^{(1)}_i{+}y^{(2)}_i}{2} \right\rfloor.
\]
By substituting $z_i^{(1)}{+}z_i^{(2)}=(x_i^{(1)}{+}x_i^{(2)})-(y_i^{(1)}{+}y_i^{(2)})$ into
\[
\min\{z_i^{(1)}, z_i^{(2)}\} \leq 
 \frac{z_i^{(1)}{+}z_i^{(2)}}{2} \leq \max\{z_i^{(1)}, z_i^{(2)}\},
\]
we obtain
\[
\min\{z_i^{(1)}, z_i^{(2)}\} + \frac{y_i^{(1)}{+}y_i^{(2)}}{2}
\leq \frac{x_i^{(1)}{+}x_i^{(2)}}{2}
\leq \max\{z_i^{(1)}, z_i^{(2)}\} + \frac{y_i^{(1)}{+}y_i^{(2)}}{2}.
\]
In the above inequalities, we round up and round down every terms, 
to obtain
\begin{align*}
\min\{z_i^{(1)}, z_i^{(2)}\}+\left\lceil\frac{y_i^{(1)}{+}y_i^{(2)}}{2}\right\rceil
&\leq \left\lceil\frac{x_i^{(1)}{+}x_i^{(2)}}{2}\right\rceil
\leq \max\{z_i^{(1)}, z_i^{(2)}\}+\left\lceil\frac{y_i^{(1)}{+}y_i^{(2)}}{2}\right\rceil, \\
\min\{z_i^{(1)}, z_i^{(2)}\}+\left\lfloor\frac{y_i^{(1)}{+}y_i^{(2)}}{2}\right\rfloor
&\leq \left\lfloor\frac{x_i^{(1)}{+}x_i^{(2)}}{2}\right\rfloor
\leq \max\{z_i^{(1)}, z_i^{(2)}\}+\left\lfloor\frac{y_i^{(1)}{+}y_i^{(2)}}{2}\right\rfloor.
\end{align*}
Thus, (\ref{ddmc-sepa-conv-cl-ineq2}) and (\ref{ddmc-sepa-conv-cl-ineq3}) are 
satisfied.

Case 2 ($x^{(1)}_i< x^{(2)}_i$ and $y^{(1)}_i< y^{(2)}_i$). 
In the same way as Case 1, we can show 
(\ref{ddmc-sepa-conv-cl-ineq2}) and (\ref{ddmc-sepa-conv-cl-ineq3}).

Case 3 ($x^{(1)}_i\ge x^{(2)}_i$ and $y^{(1)}_i< y^{(2)}_i$).
In this case, we have
\[
z^\prime_i=\left\lceil \frac{x^{(1)}_i{+}x^{(2)}_i}{2} \right\rceil - 
\left\lfloor \frac{y^{(1)}_i{+}y^{(2)}_i}{2} \right\rfloor , \quad
z^{\prime\prime}_i=\left\lfloor \frac{x^{(1)}_i{+}x^{(2)}_i}{2} \right\rfloor - 
\left\lceil \frac{y^{(1)}_i{+}y^{(2)}_i}{2} \right\rceil.
\]
Moreover, we obtain
\begin{align*}
z_i^{(1)}=x_i^{(1)}-y_i^{(1)}&\geq \left\lceil\frac{x^{(1)}_i{+}x^{(2)}_i}{2} \right\rceil - 
\left\lfloor \frac{y^{(1)}_i{+}y^{(2)}_i}{2} \right\rfloor = z^{\prime}_i \\
&\geq \left\lfloor \frac{x^{(1)}_i{+}x^{(2)}_i}{2} \right\rfloor - 
\left\lceil \frac{y^{(1)}_i{+}y^{(2)}_i}{2}\right\rceil = z^{\prime\prime}_i \\
&\ge x_i^{(2)}-y_i^{(2)}
=z_i^{(2)},
\end{align*}
and hence, (\ref{ddmc-sepa-conv-cl-ineq2}) and (\ref{ddmc-sepa-conv-cl-ineq3}).

Case 4: ($x^{(1)}_i< x^{(2)}_i$ and $y^{(1)}_i\ge y^{(2)}_i$).
In the same way as Case 3, we can show 
(\ref{ddmc-sepa-conv-cl-ineq2}) and (\ref{ddmc-sepa-conv-cl-ineq3}).
(End of the proof of Claim).
\end{proof}

\begin{corollary}
Minkowski sum of a DDM-convex set and an integral box is also DDM-convex,
where an integral box is the set defined by 
$\{ x \in \ZZ^n \mid a \leq x \leq b\}$ for some $a \in (\ZZ \cup \{-\infty\})^n$ 
and $b \in (\ZZ \cup \{+\infty\})^n$ with $a \leq b$.
\end{corollary}

\section{Proximity theorems}\label{proximity}

For a function $f:\ZZ^n \to \RR \cup \{+\infty\}$ and a positive integer $\alpha$,
a proximity theorem estimates the distance between a given local minimizer
$x^\alpha$ of the $\alpha$-scaled function $f^\alpha$ and a minimizer $x^*$ of $f$.
For instance, the following proximity theorems for L$^\natural$-convex functions and
globally/locally discrete midpoint convex functions are known.

\begin{theorem}[\cite{IS2002,DCA}]
Let $f:\ZZ^n \to \RR \cup \{+\infty\}$ be an L$^\natural$-convex function,
$\alpha$ be a positive integer and $x^\alpha \in \dom f$.
If $f(x^\alpha)\le f(x^\alpha+\alpha d)$ for all $d\in\{0, +1\}^n \cup \{0,-1\}^n$, then
there exists $x^*\in\argmin f$ with $\|x^\alpha-x^*\|_\infty\le n(\alpha -1)$.
\end{theorem}

\begin{theorem}[\cite{MMTT-DMC}]
Let $f:\ZZ^n \to \RR \cup \{+\infty\}$ be 
a globally/locally discrete midpoint convex function,
$\alpha$ be a positive integer and $x^\alpha \in \dom f$.
If $f(x^\alpha)\le f(x^\alpha+\alpha d)$ for all $d\in\{-1, 0, +1\}^n$, then
there exists $x^*\in\argmin f$ with $\|x^\alpha-x^*\|_\infty\le n(\alpha -1)$.
\end{theorem}

In the same way as the arguments in \cite{MMTT-DMC}, we can show
the following proximity theorem for DDM-convex functions.

\begin{theorem}\label{thm:proximity}
Let $f:\ZZ^n \to \RR \cup \{+\infty\}$ be a DDM-convex function, 
$\alpha$ be a positive integer and $x^\alpha\in \dom f$.
If $f(x^\alpha)\le f(x^\alpha+\alpha d)$ for all $d\in\{-1, 0, +1\}^n$, then
there exists $x^*\in\argmin f$ with $\|x^\alpha-x^*\|_\infty\le n(\alpha -1)$.
\end{theorem}
We note that $f^\alpha$ is also DDM-convex by Theorem~\ref{thm:scaling}
and $x^\alpha$ corresponds to a minimizer $\veczero$ of 
$f^\alpha(y) = f(x^\alpha+\alpha y)$ by Corollary~\ref{col:1-opt}.
We emphasize that the bound $n(\alpha-1)$ for DDM-convex functions 
is the same as that for  L$^\natural$-convex functions and
globally/locally discrete midpoint convex functions.

To prove Theorem~\ref{thm:proximity}, we assume $x^\alpha = \veczero$ 
without loss of generality.
Let $S=\{x\in\ZZ^n \mid \|x\|_\infty\leq n(\alpha -1)\}$,
$W=\{x\in\ZZ^n \mid \|x\|_\infty=n(\alpha-1)+1\}$ and let 
$\gamma=\argmin\{f(x) \mid x\in S\}$.
We show that
\begin{equation}\label{boxbarrier}
f(y)\geq\gamma\qquad(\forall y\in W).
\end{equation}
Then Theorem~\ref{box-barrier} (box-barrier property) implies that
$f(z)\geq\gamma$ for all $z\in\ZZ^n$.

Fix $y =(y_1,\ldots,y_n)\in W$, and let $\|y\|_\infty=m (=n(\alpha-1)+1)$.
By using
\[
  A_k=\{i \mid y_i\ge k\}, \quad B_k=\{i \mid y_i\le -k\}\qquad (k=1, \dots, m),
\]
we can write $y$ as
\[
  y=\sum_{k=1}^m (\vecone_{A_k}-\vecone_{B_k}),
\]
where $A_1\supseteq \cdots \supseteq A_m$, 
$B_1\supseteq \cdots \supseteq B_m$, 
$A_1\cap B_1=\emptyset$ and $A_m\cup B_m\not=\emptyset$.

\begin{lemma}\label{lem:proximity}
There exists some $k_0\in\{1, \dots, m-\alpha+1\}$ with
$(A_{k_0}, B_{k_0})=(A_{k_0+j}, B_{k_0+j})$ for $j=1, \ldots, \alpha-1$.
\end{lemma}
\begin{proof}
By $A_m \cup B_m \neq\emptyset$, we may assume $A_m\not=\emptyset$.
Let  $s=|A_1|$ and $(a_k, b_k)=(|A_k|, |B_k|+s)$ for $k=1, \dots, m$.
Since $A_1\supseteq \cdots \supseteq A_m \neq \emptyset$, 
$B_1\supseteq \cdots \supseteq B_m$ and 
$A_1\cap B_1=\emptyset$, we have $n-s \geq |B_1|$, 
$s= a_1\geq \cdots\geq a_m \geq 1$ and $n \geq b_1 \geq \cdots \geq b_m\geq s$.
Therefore, $(s, n) \geq (a_1, b_1) \geq \cdots\geq (a_m, b_m) \geq (1, s)$.
Because $m=n(\alpha-1)+1$ and the length of a strictly decreasing chain connecting 
$(s, n)$ to $(1, s)$ in $\ZZ^2$ is bounded by $n$, 
there exists a constant subsequence of length $\geq \alpha$ 
in the sequence $\{(a_k,b_k)\}_{k=1,\ldots,m}$ by the pigeonhole principle.
Hence the assertion holds.
\end{proof}

By using $k_0$ in Lemma~\ref{lem:proximity}, 
we define a subset  $J$ of $\{1, \dots, m\}$ by
$J=\{k_0, \dots, k_0+\alpha-1\}$.
By the parallelogram inequality (\ref{parallelo-equ}) in 
Theorem~\ref{thm:ddmcparalleloineq2}, where 
$d_0=\vecone_{A_{k_0}}-\vecone_{B_{k_0}}$ and
$d = \sum_{j\in J}(\vecone_{A_j}-\vecone_{B_j})=\alpha d_0$, we obtain
\[
   f(\veczero)+f(y)\geq f(\alpha d_0) + f(y - \alpha d_0).
\]
By the assumption, we have $f(\alpha d_0)\geq f(x^\alpha) = f(\veczero)$.
We also have $y - \alpha d_0 \in S$ because
\[
\|y - \alpha d_0\|_\infty=m-\alpha=(n-1)(\alpha-1)\leq n(\alpha -1).
\]
By the definition of $\gamma$, $f(y - \alpha d_0)\geq\gamma$ must hold.
Therefore,
\[
 f(y) \geq f(y - \alpha d_0)+[f(\alpha d_0)-f(\veczero)] \geq \gamma+0 = \gamma,
\]
which implies (\ref{boxbarrier}), completing the proof of Theorem~\ref{thm:proximity}.

\section{Minimization Algorithms}\label{min-algo}

In this section, we propose two algorithms for 
DDM-convex function minimization.

\subsection{The 1-neighborhood steepest descent algorithm}
\label{sec:descentalgo}

We first propose a variant of steepest descent algorithm 
for DDM-convex function minimization problem.

Let $f:\ZZ^n \to \RR \cup \{+\infty\}$ be a DDM-convex function with
$\argmin f\not=\emptyset$.
We suppose that an initial point
\[
x^{(0)}\in\dom f \setminus \argmin f 
\]
is given.
Let $L$ denote the minimum $l_\infty$-distance between $x^{(0)}$ and a minimizer
of $f$, that is, $L$ is defined by
\[
L=\min\{\|x-x^{(0)}\|_\infty \mid x\in\argmin f\}.
\]
For all $k=0, 1, \dots, L$ we define sets $S_k$ by
\[
S_k=S_k(x^{(0)})=\{x\in \ZZ^n \mid \|x-x^{(0)}\|_\infty\leq k\}.
\]
The idea of our algorithm is to generate a sequence of minimizers in $S_k$ for
$k=1,\ldots,L$.
The next proposition guarantees that consecutive minimizers can be chosen to be
close to each other.

\begin{proposition}\label{prop:1-neighbor}
For each $k = 1,\ldots,L$ and for any $x^{(k-1)}\in\argmin\{f(x) \mid x\in S_{k-1}\}$,
there exists $x^{(k)}\in\argmin\{f(x) \mid x\in S_k \}$ with 
$\|x^{(k)}-x^{(k-1)}\|_\infty\leq 1$.
\end{proposition}
\begin{proof}
If $k=1$, the assertion is obvious.
Suppose that $k \geq 2$ and $y$ is any point in $S_k$.
By (\ref{DDMC}) for $x^{(k-1)}$ and $y$, we have
\begin{equation}\label{prop:1-neighbor-eq1}
f(x^{(k-1)})+f(y)\geq f(\dmid(x^{(k-1)}, y))+f(\dmid(y, x^{(k-1)})).
\end{equation}
Since $x^{(k-1)}, y\in S_k$ and $S_k$ is a DDM-convex set, we also have
\begin{equation}\label{prop:1-neighbor-eq2}
\dmid(y, x^{(k-1)})\in S_k.
\end{equation}
Next, we show
\begin{equation}\label{prop:1-neighbor-eq3}
\dmid(x^{(k-1)}, y)\in S_{k-1}.
\end{equation}
To show this we arbitrarily fix $i\in\{1, \dots, n\}$, and consider the
two cases: Case~1:  $x^{(k-1)}_i-y_i=l\ (l\ge1)$ and Case~2: $x^{(k-1)}_i-y_i=-l\ (l\ge1)$.

Case1 ($x^{(k-1)}_i-y_i=l\ (l\ge1)$).
In this case, $\dmid(x^{(k-1)}, y)_i=x^{(k-1)}_i-\left\lfloor\frac{l}{2}\right\rfloor$
and $\left\lfloor\frac{l}{2}\right\rfloor\leq l-1$.
Thus, we have
\begin{align*}
x_i^{(0)}+(k-1) &\geq x_i^{(k-1)} \\
 &\geq \dmid(x^{(k-1)}, y)_i =x^{(k-1)}_i-\left\lfloor\frac{l}{2}\right\rfloor \\
  &\geq x^{(k-1)}_i-(l-1)=y_i+1\geq x_i^{(0)}-(k-1).
\end{align*}

Case2 ($x^{(k-1)}_i-y_i=-l\ (l\ge1)$).
In this case, $\dmid(x^{(k-1)}, y)_i=x^{(k-1)}_i+\left\lfloor\frac{l}{2}\right\rfloor$
and $\left\lfloor\frac{l}{2}\right\rfloor\leq l-1$.
Thus, we have
\begin{align*}
x_i^{(0)}-(k-1) &\leq x^{(k-1)}_i \\
 &\leq \dmid(x^{(k-1)}, y)_i=x^{(k-1)}_i+\left\lfloor\frac{l}{2}\right\rfloor \\
 &\leq x^{(k-1)}_i+(l-1)=y_i-1\le x_i^{(0)}+(k-1).
\end{align*}
By the above arguments, (\ref{prop:1-neighbor-eq3}) holds.

Let $y^*$ be a point $y$ in $\argmin\{f(x) \mid x\in S_k\}$ minimizing 
$\|y-x^{(k-1)}\|_\infty$.
To prove $\|y^*-x^{(k-1)}\|_\infty\leq 1$ by contradiction, suppose that
$\|y^*-x^{(k-1)}\|_\infty\geq 2$, which yields
$\|y^*-x^{(k-1)}\|_\infty >\|\dmid(y^*, x^{(k-1)})-x^{(k-1)}\|_\infty $.
Since $\dmid(y^*, x^{(k-1)}) \in S_k$ by (\ref{prop:1-neighbor-eq2}), this implies
$f(y^*)<f(\dmid(y^*, x^{(k-1)}))$.
Moreover, by (\ref{prop:1-neighbor-eq3}), we have
$f(x^{(k-1)})\leq f(\dmid(x^{(k-1)}, y^*))$.
These two inequalities contradict (\ref{prop:1-neighbor-eq1}) for
$x^{(k-1)}$ and $y^*$.
Hence $\|y^*-x^{(k-1)}\|_\infty\leq 1$ must hold.
\end{proof}

By Proposition~\ref{prop:1-neighbor}, it seems be natural to assume that
we can find a minimizer of $f$ within the {\em 1-neighborhood} $N_1(x)$ of 
$x$ defined by
\[
  N_1(x) = \{ y \in \ZZ^n \mid \|z - x\|_\infty \leq 1\}.
\]
With the use of a {\em 1-neighborhood minimization oracle}, which finds 
a point minimizing $f$ in $N_1(x)$ for any $x \in \dom f$,  
our algorithm can be described as below.

\begin{tabbing}     
\= {\bf The 1-neighborhood steepest descent algorithm}%
\\
\> \quad  D0: 
   \= Find  $x\sp{(0)} \in \dom f$, and set $k:= 1$.
\\
\> \quad  D1:
   \>  Find $x\sp{(k)}$  that minimizes $f$ in $N_{1}(x\sp{(k-1)})$.
\\
\> \quad  D2: 
    \> If $f(x\sp{(k)}) = f(x\sp{(k-1)})$, then 
        output $x\sp{(k-1)}$ and  stop.             
\\
\> \quad  D3: 
  \> Set  $k := k+1$, and go to D1.  
\end{tabbing}

\begin{theorem}\label{thm:sdalgo}
For a DDM-convex function $f:\ZZ^n \to \RR \cup \{+\infty\}$ with
$\argmin f\not=\emptyset$, the 1-neighborhood steepest descent algorithm
finds a minimizer of $f$ exactly in $(L{+}1)$ iterations, that is,
exactly in $(L{+}1)$ calls of the 1-neighborhood minimization oracles.
\end{theorem}
\begin{proof}
By Proposition~\ref{prop:1-neighbor}, the sequence $\{x^{(k)}\}$ generated by
the 1-neighborhood steepest descent algorithm satisfy
\begin{align}\label{sdalgo1}
x^{(k)}\in\argmin\{f(x) \mid x\in S_k\} \qquad  (k=1, 2, \dots).
\end{align}

\smallskip
\noindent \textbf{Claim:} If $f(x^{(k)})=f(x^{(k-1)})$ at Step D2, then 
$x^{(k-1)}\in\argmin f$.

\smallskip\noindent(Proof) 
For any $d\in\{+1, 0, -1\}^n$, $x^{(k-1)}+d$ belongs to $S_{k}$, and hence,
$f(x^{(k)})\leq f(x^{(k-1)}+d)$ by (\ref{sdalgo1}).
Therefore, if $f(x^{(k)})=f(x^{(k-1)})$, then $f(x^{(k-1)})\leq f(x^{(k-1)}+d)$ 
for any $d$.
Corollary~\ref{col:1-opt} in Section~\ref{relation-sec} guarantees 
$x^{(k-1)}\in \argmin f$. (End of the proof of Claim).

By the definition of $L$,  $x^{(k)} \neq x^{(k-1)}$ if $k \leq L$, and $x^{(L)}=x^{(L+1)}$.
Therefore our algorithm stops in $(L+1)$ iterations.
\end{proof}

\begin{remark}
Theorem~\ref{thm:sdalgo} says that the sequence of points generated by
the 1-neighborhood steepest descent algorithm is bounded by
the $\ell_\infty$-distance between an initial point and 
the nearest minimizer.
Similar facts are pointed out for L$^\natural$-convex function
minimization~\cite{KS09lnatmin,MS14exbndLmin,Shi17L} and 
globally/locally discrete midpoint convex function
minimization~\cite{MMTT-DMC}.
\hfill\finbox
\end{remark}

\begin{remark}
Let $F(n)$ denote the number of function evaluations in
the 1-neighborhood minimization oracle.
Since any function defined on $\{0,1\}^n$ is DDM-convex, 
the 1-neighborhood minimization problem is NP-hard.
In almost cases, $F(n)$ seems to be $\Theta(3^n)$ 
by a brute-force calculation, because $|N_1(\cdot)| = 3^n$.
Fortunately, for L$^\natural$-convex functions, $F(n)$ is bounded
by a polynomial in $n$.
Another hopeful case is a fixed parameter tractable case, that is,
the case where there exists some parameter $k$ such that
the number of function evaluations $F(n,k)$ in $n$ and $k$ is bounded by 
a polynomial $p(n)$ in $n$ times any function $g(k)$ in $k$ (see the next remark).
\hfill\finbox
\end{remark}

\begin{remark}
Let us consider the following problem:
\[
 \begin{array}{|lr}
  \mbox{ minimize} & \frac{1}{2} x^\top Q x + c^\top x \\
  \mbox{ subject to} & x \in \ZZ^n,
 \end{array}
\]
where a symmetric matrix $Q \in \RR^{n \times n}$ is nonsingular and 
diagonally dominant with nonnegative diagonals, and $c \in \RR^n$.
Since $Q$ is nonsingular, the (convex) continuous relaxation problem
has a unique minimizer $-Q^{-1}c$.
Furthermore, because the objective function is 2-separable convex,
it follows from Theorem~\ref{thm:real2discrete1} in the next section 
that there exists an optimal solution in the box:
\[
 B = \{ x \in \ZZ^n \mid -Q^{-1}c - n\vecone \leq x \leq
  -Q^{-1}c + n\vecone\}.
\]
Therefore, the 1-neighborhood steepest descent algorithm with
an initial point $\lfloor -Q^{-1}c \rfloor$ find an optimal solution
in $O(n)$ iterations.
Furthermore, if $Q = [q_{ij}]$ is $(2k+1)$-diagonal, that is,
\[
 q_{ij} = 0 \qquad (i=1,\ldots,n;\; j:|j-i| > k),
\]
then $F(n,k) = O(n 3^{k+1})$ as in \cite{GuCuiPeng}.
\hfill\finbox
\end{remark}

\subsection{Scaling algorithm}

In the same way as the scaling algorithm for minimization of
globally/locally discrete midpoint convex functions in \cite{MMTT-DMC}, 
the scaling property (Theorem~\ref{thm:scaling}) and 
the proximity theorem (Theorem~\ref{thm:proximity}) enable us to design
a scaling algorithm for the minimization of DDM-convex functions with
bounded effective domains.

Let $f:\ZZ^n \to \RR \cup \{+\infty\}$ be a DDM-convex function with
bounded effective domain.
We suppose that $K_\infty=\max\{\|x-y\|_\infty \mid x, y\in\dom  f\} \;
(K_\infty<+\infty)$ and an initial point $x \in \dom f$ are given.
Our algorithm can be described as follows.

\begin{tabbing}     
\= {\bf Scaling algorithm for DDM-convex functions}%
\\
\> \quad  S0: 
   \= Let $x \in \dom f$ and $\alpha := 2\sp{\lceil \log_{2} (K_{\infty}+1) \rceil}$.
\\
\> \quad  S1:
   \>  Find a vector $y$  that minimizes $f\sp{\alpha}(y) =   f(x{+}\alpha y)$ subject to
    $\| y \|_{\infty} \leq n$ \\
   \>\> (by the 1-neighborhood steepest descent algorithm), and \\
   \>\> set $x:= x+ \alpha y$.  \\
\> \quad  S2: 
    \> If $\alpha = 1$, then stop \
       ($x$ is a minimizer of $f$).             
\\
\> \quad  S3: 
  \> Set  $\alpha:=\alpha/2$, and go to S1.  
\end{tabbing}

\begin{theorem}
For a DDM-convex function $f:\ZZ^n \to \RR \cup \{+\infty\}$ 
with bounded effective domain, the scaling algorithm finds a minimizer of $f$
in $O(n\log_2 K_\infty)$ calls of the 1-neighborhood minimization oracles.
\end{theorem}
\begin{proof}
The correctness of the algorithm can be shown by induction on $\alpha$.
If $\alpha = 2\sp{\lceil \log_{2} (K_{\infty}+1) \rceil}$, then
$x$ is a unique point of $\dom f^\alpha$ because 
$\alpha=2^{\lceil\log_2 (K_\infty+1)\rceil}>K_\infty$, that is, 
a minimizer of $f^\alpha$.
Let $x^{2\alpha}$ denote the point $x$ at the beginning of S1 for $\alpha$ and assume
that $x^{2\alpha}$ is a minimizer of $f^{2\alpha}$.
The function $f^\alpha(y)=f(x^{2\alpha}+\alpha y)$ is DDM-convex by
Theorem~\ref{thm:scaling}.
Let $y^\alpha=\argmin \{f^\alpha(y) \mid \|y\|_\infty\le n\}$
and $x^\alpha=x^{2\alpha}+\alpha y^\alpha$.
Theorem~\ref{thm:proximity} guarantees that $x^\alpha$ is a minimizer of $f^\alpha$
because of  $x^{2\alpha} \in \argmin f^{2\alpha}$.
At the termination of the algorithm, we have $\alpha=1$ and $f^\alpha=f$.
The output of the algorithm, which is computed by 
the 1-neighborhood steepest descent algorithm, satisfies 
the condition of Corollary~\ref{col:1-opt}, and hence, 
the output is indeed a minimizer of $f$.

The time complexity of the algorithm can be analyzed as follows: 
by Theorem~\ref{thm:sdalgo}, S1 terminates in $O(n)$ calls of 
the 1-neighborhood minimization oracles in each iteration.
The number of iterations is $O(\log_2 K_\infty)$.
Hence, the assertion holds.
\end{proof}

\section{DDM-convex functions in continuous variables}\label{continuousDDMC}

In \cite{continuousL}, proximity theorems between 
L$^\natural$-convex functions and their continuous relaxations are proposed.
We extend these results to DDM-convexity.

It is known that the continuous version of L$^\natural$-convexity can naturally 
be defined by using translation-submodularity (\ref{trans-submo}).
In this section, we define DDM-convexity in continuous variables in a different way.
We call a continuous convex function 
$F : \RR^n \to \RR \cup\{+\infty\}$ a 
{\em directed discrete midpoint convex function in continuous variables
($\RR$-DDM-convex function)}
if for any positive integer $\alpha$, the function 
$f^{1/\alpha} : \ZZ^n \to \RR \cup\{+\infty\}$ defined by
\begin{equation}\label{fracscaling}
 f^{1/\alpha}(x)=F\left(\frac{x}{\alpha}\right)
   \qquad(x\in\ZZ^n)
\end{equation}
is DDM-convex.
We denote by $f$ the DDM-convex function $f^{1/1}$
which is nothing but the restriction of $F$ to $\ZZ^n$.

An example of an $\RR$-DDM-convex function is 
a continuous 2-separable convex function $F$ which is defined as
\[
 F(x)=\sum_{i=1}^n \xi_i(x_i)+
 \sum_{i,j:j\not=i}\varphi_{ij} (x_i-x_j)+
 \sum_{i,j:j\not=i} \psi_{ij}(x_i+x_j)\qquad (x\in\RR^n)
\]
for univariate continuous convex functions $\xi_{i}, \varphi_{ij}, \psi_{ij} : 
\RR \to \RR \cup \{+\infty\}\; 
(i=1,\ldots,n; j \in \{1,\ldots,n\} \setminus \{i\}$) as below.
The restriction $f$ of $F$ to $\ZZ^n$ is trivially a 2-separable convex function
on $\ZZ^n$ defined by (\ref{def:2separable}).
Furthermore, the function $F^{1/\alpha}:\RR^n \to \RR \cup\{+\infty\}$ 
defined by
\[
 F^{1/\alpha}(x) = F\left(\frac{x}{\alpha}\right)
 \qquad (x \in \RR^n)
\]
is also a continuous 2-separable convex function, and hence, the restriction
$f^{1/\alpha}$ of $F^{1/\alpha}$ to $\ZZ^n$ is also 
a 2-separable convex function on $\ZZ^n$.

We have the following proximity theorems between
an $\RR$-DDM-convex function $F$ and its restriction $f$ to $\ZZ^n$.

\begin{theorem}\label{thm:discrete2real}
Let $F:\RR^n\to\RR\cup\{+\infty\}$ be an $\RR$-DDM-convex function.
For each $x^* \in\argmin f\;(= \argmin f^{1/1})$, there exists 
$\overline{x}\in\argmin F$ with 
$\|x^*-\overline{x}\|_\infty\leq n$.
\end{theorem}
\begin{proof}
Since $f$ is DDM-convex, by Corollary~\ref{col:1-opt} in 
Section~\ref{relation-sec}, we have
\[
f(x^*)\le f(x^*+d) \qquad (\forall d\in\{-1, 0, +1\}^n).
\]
Thus, for every integer $\alpha\geq 2$, by
$f(x)=f^{1/\alpha}(\alpha x)\; (x\in\ZZ^n)$, we have
\[
f^{1/\alpha}(\alpha x^*)\leq
 f^{1/\alpha}(\alpha x^*+\alpha d)\qquad
 (\forall d\in\{-1, 0, +1\}^n).
\]
By Theorem~\ref{thm:proximity} for $f^{1/\alpha}$, 
there exists $x^{1/\alpha} \in \argmin f^{1/\alpha}$
with
\begin{equation}\label{thm:discrete2real-eq1}
\alpha x^*-(\alpha-1)n\vecone \leq x^{1/\alpha}
 \leq \alpha x^*+(\alpha-1)n\vecone.
\end{equation}
By dividing all terms in (\ref{thm:discrete2real-eq1}) by $\alpha$,
we obtain
\[
x^*-n\vecone \leq x^*-\frac{\alpha-1}{\alpha}n\vecone \leq
 \frac{x^{1/\alpha}}{\alpha} \leq
  x^*+\frac{\alpha-1}{\alpha}n\vecone \leq x^*+n\vecone.
\]
Let $B=\{x\in\RR^n \mid x^*-n\vecone \leq x \leq x^*+n\vecone\}$.
For each integer $k\geq 1$, considering $\alpha_k=2^k$ and 
$x^{1/\alpha_k}\in\argmin f^{1/\alpha_k}$,
we have $\frac{x^{1/\alpha_k}}{\alpha_k}\in B$.
Since $B$ is compact, there exists a subsequence 
$\{\frac{x^{1/\alpha_{k_i}}}{\alpha_{k_i}}\}$ with
\[
\lim_{i\to \infty}\frac{x^{1/\alpha_{k_i}}}{\alpha_{k_i}}=\overline{x}\in B.
\]
Since $F$ is continuous, we have
\[
 \lim_{i\to\infty}F(
   \frac{x^{1/\alpha_{k_i}}}{\alpha_{k_i}})
   =F(
    \lim_{i\to\infty}\frac{x^{1/\alpha_{k_i}}}{\alpha_{k_i}})
   =F(\overline{x}).
\]
Since $x^{1/\alpha_{k_i}}\in\dom f^{1/\alpha_{k_{i+1}}}$
holds for each $i$ by the definition (\ref{fracscaling}), we have
$F(\frac{x^{1/\alpha_{k_1}}}{\alpha_{k_1}})
\geq F(\frac{x^{1/\alpha_{k_2}}}{\alpha_{k_2}})
\ge\cdots\ge F(\frac{x^{1/\alpha_{k_i}}}{\alpha_{k_i}})
\ge \cdots$ which together with 
$x^{1/2^{k_i}}\in\argmin f^{1/2^{k_i}}$ for all $i$,
guarantees that
\begin{equation}\label{thm:discrete2real-eq2}
F(\overline{x} )\le F(\frac{x^{1/2^{k_i}}}{2^{k_i}})
  =\min f^{1/2^{k_i}} \qquad (i=1,2,\ldots).
\end{equation}

We finally show $F(\overline{x})=\min F$, that is,
$\overline{x} \in\argmin F$.
Suppose to the contrary that there exists $x'$ with
$F(x') < F(\overline{x})$.
Let $\varepsilon = F(\overline{x})-F(x') > 0$.
By the continuity of $F$, there exists $\delta_{\varepsilon}$
such that
\begin{equation}\label{thm:discrete2real-eq3}
  \forall y \in \RR^n,\;\|x'-y\|_\infty <\delta_{\varepsilon} \;\Rightarrow\;
  |F(x')-F(y)|<\varepsilon. 
\end{equation}
Because there exist $N\in\{k_i \mid i=1, 2, \dots\}$ and $y \in \RR^n$ such that
$2^N y \in \ZZ^n$ and $\|x'-y\|_\infty <\delta_{\varepsilon}$, 
by (\ref{thm:discrete2real-eq3}), we have
\[
  \min f^{1/2^N} \leq F(y)
  <F(x')+\varepsilon = F(\overline{x}),
\]
which contradicts (\ref{thm:discrete2real-eq2}).
Therefore, $\overline{x}$ must be a minimizer of $F$.
\end{proof}

If $F$ has a unique minimizer,
the converse of Theorem~\ref{thm:discrete2real} also holds.

\begin{theorem}\label{thm:real2discrete1}
Let $F:\RR^n\to\RR\cup\{+\infty\}$ be an $\RR$-DDM-convex function.
If $F$ has a unique minimizer $\overline{x}$,
there exists $x^*\in\argmin f$ with $\|x^*-\overline{x}\|_\infty\leq n$.
\end{theorem}
\begin{proof}
Let $\overline{x}$ be a unique minimizer of $F$.
If $f$ has a minimizer $x^*$, then $\|x^*-\overline{x}\|_\infty\leq n$
must hold by Theorem~\ref{thm:discrete2real}.
Thus, it is enough to show that $f$ has a minimizer.

Suppose to the contrary that $f$ has no minimizer and
let $B=\{x\in\RR^n \mid \overline{x}-n\vecone \leq x 
\leq \overline{x}+n\vecone\}$.
Then, there exists $y \in \dom f \setminus B$ such that
$f(y) < f(x)$ for all $x \in \dom f \cap B$.
Let $\ell = \|y-\overline{x}\|_\infty$ and
$B' = \{x\in\RR^n \mid \overline{x}-\ell\vecone \leq x 
\leq \overline{x}+\ell\vecone\}$.
Note that $\ell > n$ and $B' \supset B$.
Let us consider the restriction $G$ of $F$ to $B'$ defined by
\[
 G(x) = \begin{cases}
    F(x) & (x \in B') \\
    +\infty & (x \not\in B')
\end{cases}
\qquad (x \in \RR^n).
\]
Obviously, $G$ is $\RR$-DDM-convex and $\overline{x}$ is a unique minimizer of $G$.
In particular, the restriction $g$ of $G$ to $\ZZ^n$ is DDM-convex and has 
a minimizer $z$ since $B'$ is bounded.
This point $z$ does not belong to $B$ since $y \not\in B$ and $f(y) < f(x)$
for all $x \in \dom f \cap B$.
However, this contradicts Theorem~\ref{thm:discrete2real} for
$G$ and $g$.
\end{proof}

If $F$ has a bounded effective domain,
a similar statement holds.
Let $\overline{K}_{\infty}=\sup\{\|x-y\|_\infty \mid x, y\in \dom F\}$.

\begin{theorem}\label{thm:real2discrete2}
Let $F:\RR^n\to\RR\cup\{+\infty\}$ be an $\RR$-DDM-convex function.
If $\dom F$ is bounded (i.e., $\overline{K}_{\infty} < \infty$),
for each $\overline{x} \in \argmin F$, there exists
 $x^*\in\argmin f$ with $\|x^*-\overline{x}\|_\infty\leq n$.
\end{theorem}
\begin{proof}
If $\dom f=\argmin f$, the assertion holds.
In the sequel, we assume that $\dom f\not= \argmin f$.
We fix a minimizer $\overline{x}$ of $F$, arbitrarily.
For a sufficiently small $\varepsilon>0$, let us consider functions
$F_\varepsilon:\RR^n\to \RR\cup\{+\infty\}$ and
$f_\varepsilon:\ZZ^n\to\RR\cup\{+\infty\}$ defined by
\begin{align*}
&F_\varepsilon(x)=F(x)
 +\varepsilon\sum_{i=1}^n(x_i-\overline{x}_i)^2 \qquad(x\in\RR^n), \\
&f_\varepsilon(x)=F_\varepsilon(x) \qquad(x\in\ZZ^n).
\end{align*}
Function $F_\varepsilon$ has the unique minimizer
$\overline{x}$ and satisfies the conditions of
Theorem~\ref{thm:real2discrete1}, 
by Proposition~\ref{prop:basic-operations}~(4), because 
$f_\varepsilon^{1/\alpha}$ defined by (\ref{fracscaling}) for
$F_\varepsilon$ is the sum of $f^{1/\alpha}$ and
a separable convex function which are DDM-convex.
Thus, by Theorem~\ref{thm:real2discrete1} for $F_\varepsilon$
and $f_\varepsilon$, there exists 
$x^{\varepsilon}\in\argmin f_\varepsilon$ with
$\|x^{\varepsilon}-\overline{x}\|_\infty\leq n$.

Let $\beta=\min\{f(x) \mid x\in \dom f\setminus \argmin f\} > 0$.
Note that $\beta$ is well-defined by boundedness of $\dom f$.
We show that if $\varepsilon<(\beta-\min f)/(n\overline{K}_{\infty}^2)$,
then $x^{\varepsilon} \in \argmin f$.
For any $x\in\argmin f$, 
by $f_\varepsilon(x^{\varepsilon})\leq f_\varepsilon(x)$, we have
\[
  f(x^{\varepsilon})\leq f(x) + \varepsilon\sum_{i=1}^n
  \{(x_i-\overline{x}_i)^2-(x^{\varepsilon}_i-\overline{x}_i)^2\}.
\]
As $f(x^{\varepsilon})\geq f(x)$, 
$\sum_{i=1}^n
\{(x_i-\overline{x}_i)^2-(x^{\varepsilon}_i-\overline{x}_i)^2\} \geq 0$
must hold.
Therefore, we obtain
\begin{align*}
f(x^{\varepsilon}) &< f(x)+\frac{\beta-\min f}{n\overline{K}_{\infty}^2}
\sum_{i=1}^n
\{(x_i-\overline{x}_i)^2-(x^{\varepsilon}_i-\overline{x}_i)^2\}\\
 &\leq f(x)+\beta-\min f = \beta,
\end{align*}
which says $x^{\varepsilon}\in\argmin f$.
\end{proof}

\begin{remark}
There is a convex function which 
is not $\RR$-DDM-convex.
For example, for a positive definite matrix 
$Q = \left[ \begin{smallmatrix} 5 & 2 \\ 2 & 1 \end{smallmatrix} \right]$,
the function
\[
  F(x) = x^\top Q x \qquad (x \in \RR^2)
\]
is convex, but the restriction $f$ of $F$ to $\ZZ^2$ is not DDM-convex by
Theorem~\ref{thm:diagdom-DDMC}, 
and hence, $F$ is not $\RR$-DDM-convex.
\hfill\finbox
\end{remark}

\begin{remark}
The convex extension of a DDM-convex function may not be $\RR$-DDM-convex.
For example,
\[
 S = \{(1,0,0), (0,1,0), (0,0,1)\}
\]
is a DDM-convex set, and hence, its indicator function $f = \delta_S$ is
DDM-convex.
We denote the convex hull of $S$ by $\overline{S}$.
Then the convex extension $F$ of $f$ is expressed by
\[
 F(x) = \begin{cases}
  0 & (x \in \overline{S}) \\
  +\infty & (x \not\in \overline{S})
 \end{cases}
  \qquad (x \in \RR^3),
\]
and $f^{1/2}$ is given by
\[
 f^{1/2}(x) = \begin{cases}
  0 & (x \in T) \\
  +\infty & (x \not\in T)
 \end{cases}
  \qquad (x \in \ZZ^3),
\]
where $T = \{(2,0,0),(1,1,0),(0,2,0),(0,1,1),(0,0,2),(1,0,1)\}$.
The function $f^{1/2}$ is not DDM-convex, 
because for $x = (2,0,0)$ and $y = (0,1,1)$, 
we have $\dmid(x,y) = (1,0,0) \not\in T$.
\hfill\finbox
\end{remark}

\bigskip\noindent
\textbf{Acknowledgements:} 
The notion of DDM-convexity was first proposed by Fabio Tardella 
at an informal meeting of Satoko Moriguchi, Kazuo Murota, Akihisa Tamura
and Fabio Tardella in 2018.
The authors wish to express their deep gratitude to Fabio Tardella.
They also thank Kazuo Murota for discussion about the first manuscript.
His comments were helpful to improve the paper.
This work was supported by JSPS KAKENHI Grant Number JP16K00023.

\end{document}